\documentclass[11pt]{amsart}

\usepackage{amsmath, amsthm, amssymb}
\usepackage{enumitem}
\usepackage{lmodern}
\usepackage{palatino}                       
\usepackage[bigdelims,vvarbb]{newpxmath}    
\usepackage[scaled=0.95]{inconsolata}       
\usepackage[T1]{fontenc}                    

\usepackage{comment}

\usepackage[abs]{overpic}		  
\usepackage[all,cmtip]{xy}

\usepackage{xcolor}
\definecolor{indigo}{rgb}{0.29, 0.0, 0.51}  
\usepackage[colorlinks, urlcolor=indigo, linkcolor=indigo, citecolor=indigo]{hyperref}

\usepackage[hcentering, vcentering, total={5.9in, 8.2in}]{geometry}  

\usepackage{tikz}
\usetikzlibrary{trees}

\usepackage{mathtools}

\newtheorem{theorem}{Theorem}[section]
\newtheorem{lemma}[theorem]{Lemma}
\newtheorem{proposition}[theorem]{Proposition}
\newtheorem{corollary}[theorem]{Corollary}

\theoremstyle{definition}
\newtheorem{definition}[theorem]{Definition}
\newtheorem{conjecture}[theorem]{Conjecture}

\newtheorem{remark}[theorem]{Remark}

\newcommand{\F}{{\mathcal{F}}}
\newcommand{\Op}{{\mathcal{O}p}}

\newcommand{\Cont}{\operatorname{Cont}}

\newcommand{\FCont}{\operatorname{FCont}}
\newcommand{\Diff}{\operatorname{Diff}}

\newcommand{\Iso}{\operatorname{Iso}}
\newcommand{\Id}{\operatorname{Id}}

\newcommand{\Loose}{\operatorname{Loose}}

\newcommand{\U}{\operatorname{U}}
\newcommand{\SO}{\operatorname{SO}}
\newcommand{\GL}{\operatorname{GL}}

\newcommand{\Path}{\operatorname{Path}}
\newcommand{\Emb}{\operatorname{Emb}}
\newcommand{\FEmb}{\operatorname{FEmb}}
\newcommand{\ot}{\operatorname{OT}}
\newcommand{\std}{\operatorname{std}}
\newcommand{\tb}{\operatorname{tb}}
\newcommand{\rot}{\operatorname{rot}}
\newcommand{\inv}{\operatorname{inv}}

\newcommand{\rel}{\operatorname{rel}}
\newcommand{\ev}{\operatorname{ev}}
\newcommand{\image}{\operatorname{Image}}

\newcommand{\Z}{\mathbb{Z}}
\newcommand{\R}{\mathbb{R}}
\newcommand{\D}{\mathbb{D}}
\newcommand{\NS}{\mathbb{S}}
\newcommand{\CFra}{\operatorname{CFr}}
\newcommand{\CStr}{\mathcal{C}}
\newcommand{\FC}{F\mathcal{C}}

\title{Convex disks with Legendrian boundary in overtwisted contact $3$-manifolds}
\author{Dahyana Farias}
\address{Departamento de Matem\'aticas \\ Universidad Aut\'onoma de Madrid \\ Madrid \\ Spain}
\email{dahyana.farias@estudiante.uam.es \\ dahyana.farias@icmat.es}

\author{Eduardo Fern\'andez}
\address{Department of Mathematics\\ University of Georgia\\ Athens\\ GA, USA}
\email{eduardofernandez@uga.edu}

\author{Francisco Presas}
\address{Instituto de Ciencias Matem\'{a}ticas CSIC-UAM-UC3M-UCM, C. Nicol\'{a}s Cabrera, 13-15, 28049 Madrid, Spain.}
\email{fpresas@icmat.es}

\author{Guillermo S\'anchez-Arellano}
\address{Departamento de \'Algebra, Geometr\'ia y Topolog\'ia, Facultad de Ciencias Matem\'aticas, Universidad Complutense de Madrid, 
	Plaza de Ciencias 3, 28040 Madrid, Spain, and Instituto de Ciencias Matem\'{a}ticas CSIC-UAM-UC3M-UCM, C. Nicol\'{a}s Cabrera, 13-15, 28049 Madrid, Spain.}
\email{guillermo\_sanchez@ucm.es}

\begin{document}

\begin{abstract}
  We classify convex disks with a fixed characteristic foliation and Legendrian boundary, up to contact isotopy relative to the boundary, in every closed overtwisted contact $3$-manifold. This classification covers cases where the neighborhood of such a disk is tight or where the boundary violates the Bennequin-Eliashberg inequality. We show that this classification coincides with the formal one, establishing an $h$-principle for these disks. As a corollary, we deduce that the space of Legendrian unknots that lie in some Darboux ball in a closed overtwisted contact 3-manifold satisfies the $h$-principle at the level of fundamental groups. Finally, we determine the contact mapping class group of the complement of each Legendrian unknot with non-positive $\tb$ invariant in an overtwisted $3$-sphere.
\end{abstract}

\maketitle

\section{Introduction}

Overtwisted contact structures were introduced by Eliashberg in $3$ dimensions, and by Borman-Eliashberg-Murphy in all dimensions in the seminal works \cite{BEM,EliashbergOT}. Since these structures are governed by an $h$-principle, the results in the literature regarding classification problems have been scarce. Some results regarding the coarse classification of submanifolds, i.e., up to contactomorphisms smoothly isotopic to the identity, have been obtained by Chatterjee, Eliashberg-Fraser, Etnyre, Etnyre-Min-Mukherjee, and Geiges-Onaran \cite{Chatterjee,Eliashberg-Fraser,EtnyreOT,EtnyreMinMuk,GeigesOnaran,GeigesOnaran2}. On the other hand, contact isotopy questions in overtwisted contact manifolds can be intricate even for loose objects, i.e., those with an overtwisted complement, as shown independently by Chekanov and Vogel \cite{VogelOvertwisted}. Nevertheless, there have been few works dealing with classifications up to contact isotopy, such as the aforementioned \cite{VogelOvertwisted}, the works of Cahn-Chernov, Cardona-Presas, Ding-Geiges, Dymara, and Farias-Fern\'andez \cite{CahnChernov,cardonapresas,DingGeigesHandle,DymaraLegendrians,Dymara,F-F} on the classification of loose objects avoiding a fixed overtwisted disk, or the work of Fern\'andez on strongly overtwisted contact $3$-manifolds \cite{F-SOT}.

In this article, we study convex disks with fixed characteristic foliations and fixed Legendrian boundaries, up to contact isotopy, in closed overtwisted contact $3$-manifolds. In most cases, we achieve a full classification by establishing a $1$-parametric $h$-principle result. The analogous question for tight contact manifolds was studied by Colin in \cite{Colin}, with its parametrized version analyzed by Fern\'andez-Mart\'inez--Aguinaga-Presas in \cite{FMP22}. The methods employed throughout this article rely on a version of Eliashberg's $h$-principle \cite{EliashbergOT} with non-fixed overtwisted disk, combined with a careful study of isotopies of overtwisted disks.

\subsection{Main result}
Let $(M,\xi)$ be a cooriented overtwisted contact $3$-manifold
and $\D^2\subseteq (M,\xi)$ an embedded convex $2$-disk with Legendrian boundary and characteristic foliation $\mathcal{F}$ determined by $T\D^2\cap \xi$. We will say that the characteristic foliation $\F$ is tight, resp. overtwisted, whether the germ of contact structures in the neighborhoods of the disk is tight, resp. overtwisted.  We define the \em space of $\mathcal{F}$-embeddings \em as the space of 
all smooth embeddings of $\D^2\hookrightarrow (M, \xi)$ that induce the characteristic foliation $\mathcal{F}$ on $\D^2$, we also require all the embeddings to agree in a neighborhood of the Legendrian boundary $\partial \D^2$. This space is denoted by $\Emb(\D^2, \mathcal{F}, (M, \xi),\rel\partial)$. 

We will also consider the space of \emph{formal $\F$-embeddings} that are a fixed genuine $\F$-embedding near the boundary. It will be denoted by $\FEmb(\D^2,\F,(M,\xi),\rel \partial)$ and it is precisely defined in \ref{sec:$F$-emb}. Roughly speaking, a formal $\F$-embedding can be thought as a pair $(e,P_s)$, where $e:\D^2\hookrightarrow M$ is a smooth embedding and $P_s$, $s\in [0,1]$, is a homotopy of plane fields over $e(\D^2)$ so that $P_0=de(T\D^2)$ and $P_1\cap\xi$ determines the charcteristic foliation $\F$.


Recall that for a null-homologous oriented Legendrian knot $\Lambda\subseteq (M,\xi)$ with Seifert surface $\Sigma$ there are two classical invariants, namely the Thurston-Bennequin number $\tb(\Lambda)$ and the rotation number $\rot(\Lambda,[\Sigma])$. Although the Thurston-Bennequin invariant is independent of choices the rotation number could depend on the homology class of the Seifert surface $[\Sigma]\in H_2(M,\Lambda)$ whenever the Euler class $e(\xi)\in H_2(M)$ is non-zero. We refer the reader to \cite{GeigesBook} for a careful discussion on this.  In the cases in which the choice of Seifert surface is clear from the context, such as when considering the Legendrian boundary $\partial \Sigma$ of a surface $\Sigma$,  we will simply write $\rot(\Lambda)$. 

\begin{theorem}\label{thm:principal}
    Let $(M,\xi)$ be a closed overtwisted contact $3$-manifold and $\D^2\subseteq (M,\xi)$ an embedded convex disk with Legendrian boundary and characteristic foliation $\mathcal{F}$. Assume that either $\F$ is tight or $\tb(\partial\D^2)+|\rot(\partial \D^2)|>-1$. Then, the inclusion
\begin{equation}\label{eq:MainInclusion}\Emb(\D^2,\F,(M,\xi),\rel \partial)\hookrightarrow \FEmb(\D^2,\F,(M,\xi),\rel \partial)
\end{equation} induces an isomorphism at the level of path-connected components. Moreover, the first relative homotopy group is trivial: $$\pi_1 (\FEmb(\D^2,\F,(M,\xi),\rel \partial),\Emb(\D^2,\F,(M,\xi),\rel \partial))=0.$$
\end{theorem}

The proof of this result will rely on a variation of Eliashberg's overtwisted $h$-principle \cite{EliashbergOT} involving a varying overtwisted disk (Corollary \ref{cor:DisksWithOTDisk}). The core of the argument will be to show that, given a formal isotopy $e^t$, $t\in [0,1]$, between two disks $e^0,e^1\in \Emb(\D^2,\F,(M,\xi),\rel \partial)$, there exists an isotopy of overtwisted disks $\Delta^t$, $t\in[0,1]$, such that $\image(e^t)\cap \Delta^t=\emptyset$ for all $t\in[0,1]$. This result will be established in Theorem \ref{thm:IsotopyOvertwistedDisks}.

To construct the required isotopy of overtwisted disks, we will use Colin's trick \cite{Colin} and discretize the given isotopy of disks $e^t$. We will then analyze various cases, depending on whether the corresponding pinched $3$-disks are tight or overtwisted. The case where the pinched $3$-disks are tight is relatively straightforward. However, the case where the pinched 
$3$-disks are overtwisted is more intricate and may fail if the hypotheses of Theorem \ref{thm:principal} are relaxed (see Remark \ref{rmk:Counterexample} for a counterexample and Remark \ref{rmk:OTDisks} for a detailed explanation of where the argument would break). Moreover, in the overtwisted case, homotopical obstructions can arise during the process. To address these obstructions, we will critically leverage the fact that a bypass triangle modifies these homotopical obstructions in a controlled manner, as shown by Huang \cite{Huang}. See Theorem \ref{thm:Huang-hopf}.

\begin{remark}
     Theorem \ref{thm:principal} cannot be directly derived from Eliashberg's $h$-principle \cite{EliashbergOT}, as illustrated by the following example. Consider two disk embeddings $e_0, e_1\in \Emb(\D^2, \mathcal{F}, (\NS^3, \xi_{\std}),\rel\partial)$ in the standard tight $3$-sphere that cobound an embedded pinched $3$-disk $(D,\xi_{\std})$. Perform a contact connected sum $(\NS^3,\xi_{\std})\#(\NS^3,\xi_0)=:(\NS^3,\xi)$, where the connected sum region lies inside $(D,\xi_{\std})$, and $(\NS^3,\xi_0)$ is an overtwisted contact structure in the formal class of the standard one. 

    We regard $e_0$ and $e_1$ as embeddings in $(\NS^3,\xi)$. Note that they now cobound an overtwisted pinched $3$-disk $(D,\xi)$ with tight complement. The schematic representation of this situation is depicted in Figure \ref{fig:toro}. It is straightforward to see that these disks are now formally isotopic. However, there does not exist an overtwisted disk  $\Delta\subseteq (\NS^3,\xi)$ such that $e_0$ and $e_1$ are even smoothly isotopic relative to it. Therefore, Eliashberg's result cannot be directly applied. 
    
    Nonetheless, Theorem \ref{thm:principal} implies that there exists a contactomorphism $\varphi:(\NS^3,\xi)\rightarrow (\NS^3,\xi)$, contact isotopic to the identity relative to the boundary of the disks such that $\varphi\circ e_0=e_1$. This leads to the non-trivial observation that the pinched $3$-disks $(D,\xi)$ and $(\varphi(D),\xi)$, schematically depicted in Figure \ref{fig:toro}, are contact isotopic relative to the equator. 
\end{remark}

\begin{figure}[h]
    \centering
    \includegraphics[scale=0.3]{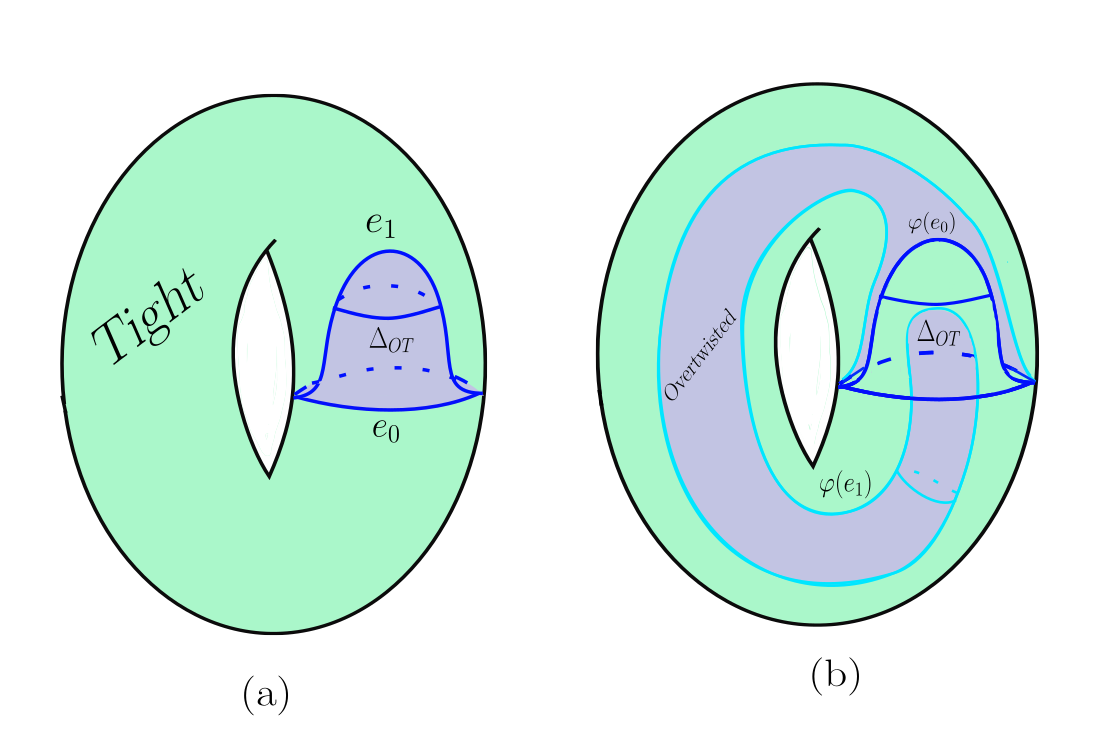}
    \caption{(a). Schematic of the two disks $e_0$ and $e_1$. The situation is described in the solid torus complementary to the boundary of the disks. The pinched $3$-disk $(D,\xi)$ is depicted in gray, the tight region is depicted in green. (b) Effect of applying the contactomorphism $\varphi$ to the pinched $3$-disk $(D,\xi)$. }
    \label{fig:toro}
\end{figure}

\begin{remark}
The work of the second author on strongly overtwisted contact $3$-manifolds addresses the cases in which $\tb(\partial \D^2)+ |\rot(\partial \D^2)| > -1$ or $(M,\xi)$ is strongly overtwisted \cite{F-SOT}. Even more, in these cases the inclusion (\ref{eq:MainInclusion}) is a weak homotopy equivalence. In this article we will address the case in which $\tb(\partial \D^2)+ |\rot(\partial \D^2)|\leq -1$ and the characteristic foliation $\F$ is tight. However, our method works in both cases and is independent of \cite{F-SOT}. 
\end{remark}

\begin{remark}\label{rmk:Counterexample}
   The hypotheses in Theorem \ref{thm:principal} are sharp. Specifically, in the case where $\tb(\partial\D^2) + |\rot(\partial\D^2)| \leq -1$ and $\F$ is overtwisted, the result does not hold in general. For instance, in an overtwisted contact $3$-manifold with disconnected space of overtwisted disks one could take a disk with tight characteristic foliation and use two transverse arcs to connect sum it with two non-isotopic overtwisted disks. Hence obtaining two disks with overtwisted characteristic foliation that are not contact isotopic and can be assumed formally isotopic. However, under the stronger assumption that the space of overtwisted disks is connected, the same conclusion of Theorem \ref{thm:principal} holds also in this case. Notice that this assumption holds for every overtwisted contact $3$-sphere but one \cite{VogelOvertwisted}, and every strongly overtwisted contact $3$-manifold \cite{F-SOT}. It is an open question if there exists an overtwisted contact $3$-manifold $(M,\xi)$ with $M\neq \NS^3$ for which the space of overtwisted disks is disconnected.
\end{remark}

\subsection{Legendrian unknots}
Our main motivation to study spaces of convex disks with Legendrian boundary was to understand spaces of Legendrian unknots in overtwisted contact $3$-manifolds. Let $\Lambda \subseteq (M,\xi)$ be an embedded Legendrian, and fix a parametrization $i:\NS^1\cong \Lambda \hookrightarrow (M,\xi)$. We will denote by $\mathcal{L}(\Lambda,(M,\xi))$ the path-connected component of $i$
in the space of Legendrian embeddings. Similarly, we will denote by $\mathcal{L}_{(p,v)}(\Lambda,(M,\xi))\subseteq \mathcal{L}(\Lambda,(M,\xi))$ the subspace of \em long \em Legendrian embeddings, i.e. Legendrian embeddings $\gamma\in \mathcal{L}(\Lambda,(M,\xi))$ so that $(\gamma(0),\gamma'(0))=(i(0),i'(0))=:(p,v)$. 

Assume that $\Lambda\subseteq (\D^3,\xi_{\std})\subseteq (M,\xi)$ is a Legendrian unknot that lies in a Darboux ball so we may find an embedded convex disk $\D^2\subseteq (\D^3,\xi_{\std})\subseteq (M,\xi)$, with tight characeristic foliation $\mathcal{F}$, such that $\partial \D^2 = \Lambda$. Then, it was observed by Fern\'{a}ndez, Mart\'{i}nez-Aguinaga and Presas in \cite{FMP22} that there is a weak homotopy equivalence 
$$ \Emb(\D^2,\mathcal{F},(M,\xi),\rel \partial) \cong \Omega \mathcal{L}_{(p,v)} (\Lambda, (M,\xi)).$$
In particular, Theorem \ref{thm:principal} allows us to compute the fundamental group of the space of Legendrian unknots that lie in a Darboux ball in every overtwisted contact $3$-manifold. 

Recall that a \em formal Legendrian embedding \em is a pair $(\gamma,F_s)$ such that \begin{itemize}
    \item $\gamma:\NS^1\hookrightarrow M$ is a smooth embedding and 
    \item $F_s:T\NS^1\rightarrow \gamma^*TM$, $s\in [0,1]$, is a homotopy of bundle monomorphisms such that $F_0=d\gamma$ and $\image(F_1)\subseteq \gamma^* \xi$. 
\end{itemize}
Notice that every Legendrian embedding $\gamma$ is naturally a formal Legendrian just by taking $F_s\equiv d\gamma$ the constant homotopy. We will denote by $F\mathcal{L}(\Lambda,(M,\xi))$ the path-connected component of the parametrization $i$ of $\Lambda$ in the space of formal Legendrians.

\begin{corollary}\label{cor:unknots} 
Let $\Lambda\subseteq (M,\xi)$ be an embedded Legendrian unknot in an overtwisted contact $3$-manifold. Assume that $\Lambda$ lies in some Darboux ball. Then, the inclusion 
\begin{equation}\label{eq:InclusionLegendrians}
\mathcal{L}(\Lambda,(M,\xi))\hookrightarrow F\mathcal{L}(\Lambda,(M,\xi))
\end{equation}
induces an isomorphism at the level of fundamental groups. 
\end{corollary}

In the case in which the underlying $3$-manifold is the $3$-sphere we precisely describe the fundamental group of these Legendrian embedding spaces.
\begin{corollary}\label{cor:UnknotsS3}
Let $\Lambda \subseteq (\NS^3,\xi)$ be an embedded Legendrian unknot is some overtwisted contact $3$-sphere that lies in some Darboux ball. Then, there is an isomorphism 
$$ \pi_1 ( \mathcal{L}(\Lambda,(\NS^3,\xi)) \cong \Z \oplus \Z, $$
in which one $\Z$-factor is generated by a full-rotation of $\Lambda$ in a Darboux ball, and the other $\Z$-factor is generated by a loop of long Legendrian unknots which are contractible as loops of smooth long unknots.  
\end{corollary}
\begin{remark}
    The situation stands in sharp contrast with the one for the tight contact $3$-sphere described in \cite{FMP22}. For instance, there does not exists a non-contractible loop of long Legendrian unknots with $\tb=-1$ in the standard tight $3$-sphere. 
\end{remark}

\subsection{The contact mapping class group of a Legendrian unknot with negative $\tb$ in an overtwisted $3$-sphere}

Denote by $\mathcal{C}_{\ot}(M)$ the space of overtwisted contact structures on a given $3$-manifold $M$. Recall that, as a consequence of Eliashberg's $h$-principle, the
isotopy classes of overtwisted contact structures on $M$ are in bijection with the homotopy classes of plane fields on $M$ \cite{EliashbergOT}. To label them in $\NS^3$ we fix the trivialization of $T\NS^3$ given by the quaternions. Then, every contact structure $\xi$ on $\NS^3$ defines a Gauss map $f_{\xi}:\NS^3\rightarrow \NS^2$, yielding to a map 
$$ \mathcal{C}_{\ot}(\NS^3)\rightarrow \operatorname{Maps}(\NS^3,\NS^2)$$ that induces an isomorphism at the level of path-connected components. In particular, $\pi_0(\mathcal{C}_{\ot}(\NS^3))\cong \pi_3(\NS^2)\cong \Z$, where the isomorphism is given by associating to each overtwisted contact structure $\xi$ the Hopf invariant $h(f_\xi)\in \Z$ of its associated Gauss map. We fix a list of representatives $\{ \xi_k: k\in \Z, h(f_{\xi_k})=k\} \subseteq \mathcal{C}_{\ot}(\NS^3)$ of these path-connected components. 

Finally, given a Legendrian $\Lambda\subseteq (\NS^3,\xi)$ we denote by $(\Op(\Lambda),\xi):=(\Op(\Lambda),\xi_{|\Op(\Lambda)})\subseteq (\NS^3,\xi)$ a standard neighborhood of $\Lambda$, which is unique up to contact isotopy fixing $\Lambda$. The Legendrian complement of $\Lambda$ is $(C(\Lambda),\xi):=(\NS^3\backslash \Op(\Lambda),\xi_{|\NS^3\backslash \Op(\Lambda)})$. 

Recall that a \em formal contactomorphism \em of a contact manifold $(M,\xi)$ is a pair $(\varphi,F_s)$, where $\varphi\in \Diff(M)$ and $F_s\in \Iso(TM,\varphi^*TM)$, $s\in [0,1]$, is a homotopy of bundle isomorphisms such that $F_0=d\varphi$ and $F_1(\xi)=\varphi^*\xi$. The group of formal contactomorphisms is denoted by $\FCont(M,\xi)$. A contactomorphism is a formal contactomorphism $(\varphi,F_s)$ such that $F_s\equiv d \varphi$. We will denote the group of contactomorphisms by $\Cont(M,\xi)\subseteq \FCont(M,\xi)$ We will always assume that our formal contactomorphisms are the indentity near the possibly empty boundary $\partial M$. 

Combining the Theorem \ref{thm:principal} with Chekanov and Vogel theorem \cite{VogelOvertwisted} about the contact mapping class group of an overtwisted sphere we conclude that

\begin{corollary}\label{cor:pi0complementounknot}
    Let $\xi\in \mathcal{C}_{\ot}(\NS^3)$ be some overtwisted contact structure on $\NS^3$ and $\Lambda\subseteq (\NS^3,\xi)$ a Legendrian unknot with $\tb(\Lambda)\leq 0$. Consider the inclusion \begin{equation}\label{eq:InclusionCont}
    i:\Cont(C(\Lambda),\xi)\hookrightarrow \FCont(C(\Lambda),\xi).
    \end{equation}
    Then, 
    \begin{itemize}
    \item [(a)] If $\tb(\Lambda)+|\rot(\Lambda)|> -1$ then the induced map $\pi_0(i)$ is an isomorphism. In particular, we have that $\pi_0(\Cont(C(\Lambda),\xi))\cong \Z\oplus \Z_2$.
    \item [(b)] If $\tb(\Lambda)+|\rot(\Lambda)|\leq -1$ then 
    \begin{itemize}
        \item [(i)] If $\xi\not\cong \xi_{-1}$ then the induced map $\pi_0(i)$ is an isomorphism. In particular, $\pi_0(\Cont(C(\Lambda),\xi))\cong \Z\oplus \Z_2$.
        \item [(ii)] If $\xi\cong \xi_{-1}$ then the induced map $\pi_0(i)$ is surjective but has order $2$ kernel. In particular, $\pi_0(\Cont(C(\Lambda),\xi))\cong \Z\oplus \Z_2\oplus \Z_2$.
    \end{itemize}
    \end{itemize}
\end{corollary}

\begin{remark}
   Case (a) in the Theorem above also follows from the work \cite{F-SOT} of the second author, in which is shown that the inclusion $i$ is a weak homotopy equivalence in this case. Our proof is independent of this.  
\end{remark}

\begin{remark}
    The ``exotic'' $\Z_2$-factor appearing in case (b)(ii) is the one found by Chekanov and Vogel \cite{VogelOvertwisted}. The fact that this exotic factor should appear is obvious since every contactomorphism of a connected contact $3$-manifold can be assumed to fix a Darboux ball and the Legendrian unknot $\Lambda$ lies in some Darboux ball by the assumption (b). The actual content of (b)(ii) is that this is the only exotic factor that appears in the complement of such a Legendrian. 
\end{remark}

\begin{remark}
    To the best of our knowledge, the examples in case (b) are the only overtwisted contact $3$-manifolds, besides the $3$-spheres treated in \cite{VogelOvertwisted}, that are not strongly overtwisted, for which there is a description of its contact mapping class group.
\end{remark}

\subsection{Further questions}

As mentioned above, our proof of Theorem \ref{thm:principal} relies on the construction of certain isotopies of overtwisted disks to enable the application of the overtwisted $h$-principle from \cite{EliashbergOT}. This is achieved by employing Colin's discretization trick \cite{Colin}, which ultimately relies on Giroux's genericity theorem \cite{GirouxConvexity,Honda}. However, since the latter does not extend to higher-dimensional families of embeddings, it remains unclear how to adapt our argument to prove Theorem \ref{thm:principal} for higher homotopy groups. Nevertheless, we believe this to be feasible:

\begin{conjecture}\label{conjectureDisk}
Assume the hypotheses of Theorem \ref{thm:principal}. Then, the inclusion (\ref{eq:MainInclusion}) is a weak homotopy equivalence.
\end{conjecture}

Similarly, for Legendrians, we propose the following:

\begin{conjecture}\label{conjectureLegendrians}
Let $\Lambda \subseteq (M,\xi)$ be an embedded Legendrian in an overtwisted contact $3$-manifold. Assume that $\Lambda$ lies in some Darboux ball. Then, the inclusion (\ref{eq:InclusionLegendrians}) is a weak homotopy equivalence.
\end{conjecture}

\begin{remark}
    Notice that Conjecture \ref{conjectureDisk} is known to be true in the cases in which $\tb(\partial \D^2)+|\rot (\partial \D^2)|>-1$ or $(M,\xi)$ is strongly overtwisted by \cite{F-SOT}. Similarly, Conjecture \ref{conjectureLegendrians} is also known to hold when $(M,\xi)$ is strongly overtwisted since in that case $\Lambda$ is strongly loose as defined in \cite{F-SOT}. 
\end{remark}

\subsection{Outline}

The article is organized as follows. In Section \ref{sec:Preliminares}, we provide a brief review of convex surface theory. Section \ref{sec:Overtwisted-h-principles} reviews Eliashberg's $h$-principle from \cite{EliashbergOT} and a variation of it with non-fixed overtwisted disks, which will be essential throughout the article (Theorem \ref{teo:LooseOTh-principle} and Corollary \ref{cor:DisksWithOTDisk}). Section \ref{sec:Creating-isotopies-of-overtwisted-disks}, the most technical part of the article, is devoted to constructing isotopies of overtwisted disks in the complement of isotopies of smooth disks with fixed boundary (Theorem \ref{thm:IsotopyOvertwistedDisks}). Section \ref{sec:ProofMainTheorem} presents the proof of Theorem \ref{thm:principal}. Finally, in Section \ref{sec:Applications}, we complete the proofs of Corollaries \ref{cor:unknots}, \ref{cor:UnknotsS3}, and \ref{cor:pi0complementounknot}.

\textbf{Acknowledgments:} The authors are grateful to Fabio Gironella for his insightful comments and questions about this work. Part of this project was carried out while GS was visiting EF at the University of Georgia in 2023, and both are thankful for the excellent work environment provided by the UGA math department. EF was partially supported by an AMS-Simons Travel Grant. GS was supported by the research programs PID2021-126124NB-I00, PID 2019-108936GB-C21 and PR27/21-029.

\section{Preliminaries}\label{sec:Preliminares}

In this Section, we briefly review Giroux's convex surface theory \cite{GirouxConvexity, Honda}. We will recall the effect of a bypass triangle on the formal class of a contact structure as described by Huang \cite{Huang}. Finally, we will review the definition of a formal $\mathcal{F}$-embedding from \cite{F-SOT}.

\subsection{Convex Surface Theory}\label{subsec:Giroux}

Let $(M, \xi)$ be a contact $3$-manifold. A vector field $X$ is called a \emph{contact vector field} if its flow preserves the contact structure $\xi$. An embedded surface $\Sigma \subseteq (M, \xi)$ is said to be \emph{convex} if there exists a contact vector field $X$ that is transverse to it.  An embedding $e: \Sigma \rightarrow (M, \xi)$ is termed \emph{convex} if its image $e(\Sigma)$ is a convex surface.
We define a surface $\Sigma$ to be \em tight \em if there exists a neighborhood of $\Sigma$ that is tight. Conversely, we say that a convex surface is \em overtwisted \em if every neighborhood of $\Sigma$ is overtwisted. Similarly, embeddings of surfaces are referred to as tight or overtwisted if their images are tight or overtwisted, respectively.

The intersection $T\Sigma \cap \xi$ defines a one-dimensional, oriented, singular foliation $\F$ on the surface $\Sigma$, referred to as the \emph{characteristic foliation}. This foliation characterizes $\xi$ near $\Sigma$: 

\begin{theorem}[Giroux]\label{teo:giroux-entorno}
    Let $\Sigma_i\subseteq (M_i,\xi_i)$ be an embedded compact surface within a contact $3$–manifold with characteristic foliation $\F_i$, for $i \in \{0,1\}$. Assume there exists a diffeomorphism $\phi: \Sigma_0 \to \Sigma_1$ such that $\phi_* \F_0=\F_1$. Then, there exists a contactomorphism $\Phi: (\Op(\Sigma_0),\xi_0) \to (\Op(\Sigma_1),\xi_1)$ such that $\Phi(\Sigma_0) = \Sigma_1$ and $\Phi_{|\Sigma_0}$ is isotopic to $\phi$ via an isotopy that preserves the characteristic foliation.   
\end{theorem}

Assume that $\Sigma \subseteq (M, \xi)$ is convex with associated with a contact vector field $X$. Then, there is a family of embedded transverse curves 
$\Gamma_X := \{p \in \Sigma : X(p) \in \xi_p\}$, known as \em dividing set. \em  The isotopy class of 
$\Gamma_X$ is independent of the choice of $X$, as the space of contact vector fields transverse to $\Sigma$ is contractible. Therefore, we will simply denote it by $\Gamma$.

The convexity of $\Sigma \subseteq (M, \xi)$ is equivalent to the existence of an \emph{$I$-invariant neighborhood} of $\Sigma$, which is contactomorphic to $(\Sigma \times I, \xi_{\mathrm{\inv}} = \ker(f dt + \beta))$, 
$t \in I := [0,1]$, for some $f \in C^\infty(\Sigma)$ and $\beta \in \Omega^1(\Sigma)$. With these coordinates, the dividing set of $\Sigma$ is then determined by $\Gamma = f^{-1}(0)$.

The following result highlights the feasibility of approximating a given surface by an arbitrarily close convex one. 

\begin{theorem}[Giroux, Honda \cite{GirouxConvexity,Honda}]\label{teo:Giroux-Honda}
Let $(M, \xi)$ be a contact manifold and $\Sigma \subset M$ a compact embedded surface with possibly empty Legendrian boundary with non-positive $\tb$. Then, there exists a $C^{\infty}$-small perturbation of $\Sigma$ ($C^0$-small near the boundary), which makes $\Sigma$ convex.
\end{theorem}

The \emph{Giroux Realization Theorem} states that for a convex surface, the contact germ is determined, up to isotopy, just by the dividing set. To clarify this statement, we introduce some notation. We denote by 
$\Emb(\Sigma, \F, (M, \xi))$ the space of embeddings 
$e : \Sigma \to M$ with characteristic foliation $\F$. Similarly, we define the space of convex embeddings $e : \Sigma \to M$ with dividing set $\Gamma$ as $\Emb(\Sigma, \Gamma, (M, \xi))$.

\begin{theorem}[Giroux Realization Theorem \cite{GirouxSurfaces}]\label{teo:realizacion-Giroux}
    Let $e:\Sigma\to (M,\xi)$ be a convex embedding with characteristic foliation $\F$ and dividing set $\Gamma$. Then, the natural inclusion
    \[
    i:\Emb(\Sigma, \F, (M,\xi))\to \Emb(\Sigma, \Gamma, (M, \xi))
    \]
    is a weak homotopy equivalence.
\end{theorem}

Finally, we have the Giroux criterion, which excludes certain configurations of dividing sets in tight contact $3$-manifolds:

\begin{theorem}[Giroux Tightness Criterion \cite{GirouxCircleBundles}]\label{teo:Criterio-Giroux}
    Let $\Sigma\subseteq (M,\xi)$ be a convex surface with possibly empty Legendrian boundary. Then, an $I$-invariant neighborhood of $\Sigma$ is tight if and only if one of the following conditions holds:
    \begin{enumerate}[label=(\roman*)]
        \item $\Sigma=\NS^2$ and the dividing set $\Gamma$ is connected, or
        \item $\Sigma\neq \NS^2$ and the dividing set does not contain homotopically trivial curves.
    \end{enumerate}
\end{theorem}

We finish by recalling the following celebrated 

\begin{theorem}[Colin \cite{Colin}]\label{Colin97}
    Let $(M,\xi)$ be a contact manifold, and let $f:\NS^2\to M$ be an embedding with convex image. If the contact manifold $(M\setminus f(\NS^2),\xi)$ is tight, then $(M,\xi)$ is also tight.
\end{theorem}

\subsection{Invariants of Legendrian knots}\label{subsec:invariantes}
The following Kanda theorem establishes fundamental relationships between the formal invariants of a Legendrian knot and the topology of its convex Seifert surface, providing key tools for understanding Legendrian knot theory through convex surface theory.
\begin{theorem}[Kanda \cite{Kanda}]\label{thm:Kanda}
    Let $\Lambda\subseteq (M,\xi)$ be an oriented Legendrian knot, $\Sigma$ a convex Seifert surface of $\Lambda$, and $\Gamma$ a dividing set of $\Sigma$. Then,
\begin{enumerate}
    \item $\tb(\Lambda)=-\frac{1}{2}\#( \Gamma\cap \Lambda)$,
    \item $\rot(\Lambda)=\chi(\Sigma_{+})-\chi(\Sigma_{-})$,
\end{enumerate}
where $\Sigma_{+}$ and $\Sigma_{-}$ are positive and negative component of $\Sigma\setminus\Gamma$, finally  $\chi(\Sigma)$ is the Euler characteristic.
\end{theorem}

The following celebrated result provides a bound on the possible formal invariants of a Legendrian knot in the tight case.

\begin{theorem}[Bennequin-Eliashberg Inequality \cite{Bennequin,EliashbergTwenty}]\label{thm:Bennequin-Eliashberg}
Let $\Lambda\subseteq (M,\xi)$ be an oriented Legendrian knot in a tight contact 3–manifold
$(M,\xi)$, and $\Sigma$ a Seifert surface of $\Lambda$. Then, the following holds
$$\tb(\Lambda) + |\rot(\Lambda )| \leq -\chi(\Sigma).$$
\end{theorem}

\subsection{Bypass basics}\label{sec:bypass-basics}

A \em bypass \em is a convex half-disk $D$ with a Legendrian boundary, formed by two Legendrian curves, 
denoted $\gamma$ and $\beta$. We refer to $\gamma$ as the \em attaching arc \em and to $\beta$ as the 
\em Honda arc\em. This half-disk contains four singularities, as shown in Figure \ref{fig:bypass1}. We say that $D$ is a \em bypass for a surface $\Sigma$ \em if $D\cap \Sigma=\gamma$. 

\begin{figure}[h]
    \centering
    \includegraphics[scale=0.3]{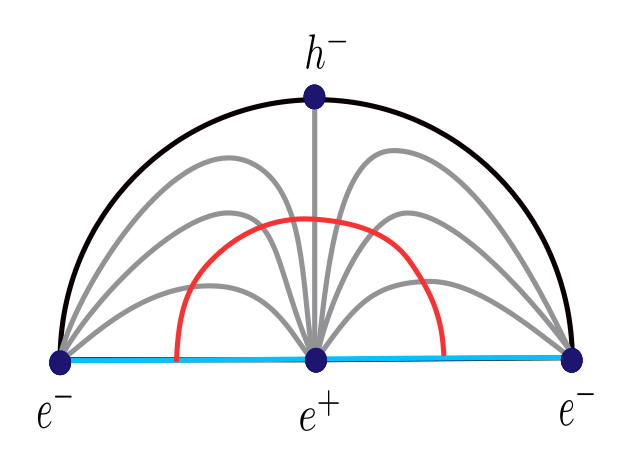}
    \caption{A positive bypass, where we can see in light blue the attaching arc, in black the Honda arc, in red the dividing set and in gray the characteristic foliation.}
    \label{fig:bypass1}
\end{figure}

The following theorem describes the interaction of a bypass when it is attached from above to a convex surface

\begin{lemma}[Bypass Attachment \cite{Honda}]\hfill\break
     Assume $D$ is a bypass for a convex surface $\Sigma$.
    Then there exists a neighborhood of $\Sigma\cup D \subset M$ diffeomorphic to $\Sigma\times[0, 1]$, such
    that $\Sigma_i = \Sigma\times\{i\}$, $i = 0, 1$, are convex, 
    $\Sigma\times[0, \varepsilon]$ is $I$–invariant, $\Sigma= \Sigma\times\{\varepsilon\}$,
    and $\Gamma_{\Sigma_1}$ is obtained from $\Gamma_{\Sigma_0}$ by performing the Bypass Attachment operation
    depicted in Figure \ref{fig:bypass-2} in a neighborhood of the attaching arc $\gamma$.
\end{lemma}

We will define a bypass $D$ for $\Sigma$ to be \em tight \em if there exists a tight neighborhood of $\Sigma\cup D$ and \em overtwisted \em in other case.

The following theorem states that the failure to convexity on $1$-parameter families of surfaces is precisely described by bypass attachments.

\begin{theorem}[Isotopy Discretization \cite{Colin,HondaGluing}]\label{thm:Discretization} 
Let $(M, \xi)$ be a contact manifold, and let $\Sigma$ and $\Sigma'$ be convex surfaces in $M$ with the same (possibly empty) Legendrian boundary. If they are smoothly isotopic (relative to the boundary), then there exists a sequence of convex surfaces $\Sigma_0 = \Sigma, \Sigma_1, \ldots, \Sigma_n = \Sigma'$ with the same boundary, such that $\Sigma_{i-1}$ and $\Sigma_i$ cobound a region $U_i \cong \Sigma \times I$ in $M$, and $(U_i, \xi)$ is contactomorphic to a single bypass attachment.
\end{theorem}

Bypasses do not always exists, however when the complement of the surface is overtwisted their existence is guaranteed: 

\begin{lemma}[\cite{Vogel}]\label{lem:bypassAttachingArc}
    If $M\setminus \Sigma$ is overtwisted and $\gamma\subset \Sigma$ is a possible attaching arc then exists bypass for $\Sigma$ with attaching arc $\gamma$.
\end{lemma}

\begin{figure}[h]
    \centering
    \includegraphics[scale=0.35]{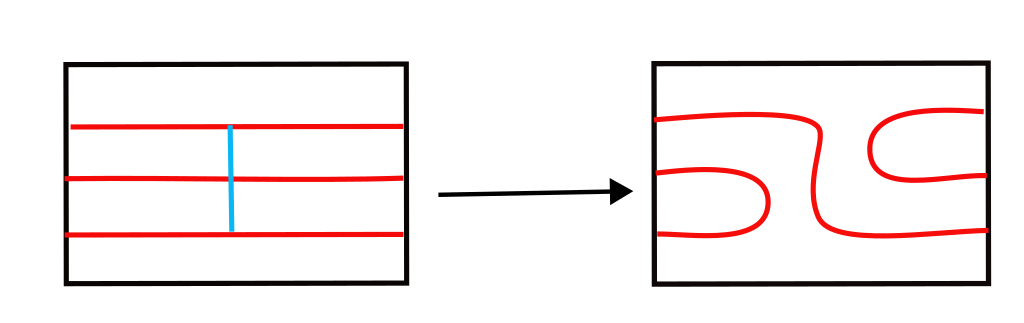}
    \caption{Result of attaching a bypass from above. The effect of attaching a bypass from below is the mirror of this figure.}
    \label{fig:bypass-2}
\end{figure}

\begin{definition}
    Let $(\Sigma\times I,\xi_{\inv})$ be an $I$-invariant contact structure and $\gamma\subseteq \Sigma$ be a Legendrian arc intersecting the dividing set in $3$ points.  A \em bypass triangle attachment along $\gamma$ \em is the composition of three bypasses
    with attaching arcs given by $\gamma$, $\gamma_1$ and $\gamma_2$, as in the Figure \ref{fig:bypass-triangle}. 
\end{definition}

\begin{figure}[h]
    \includegraphics[scale=0.28]{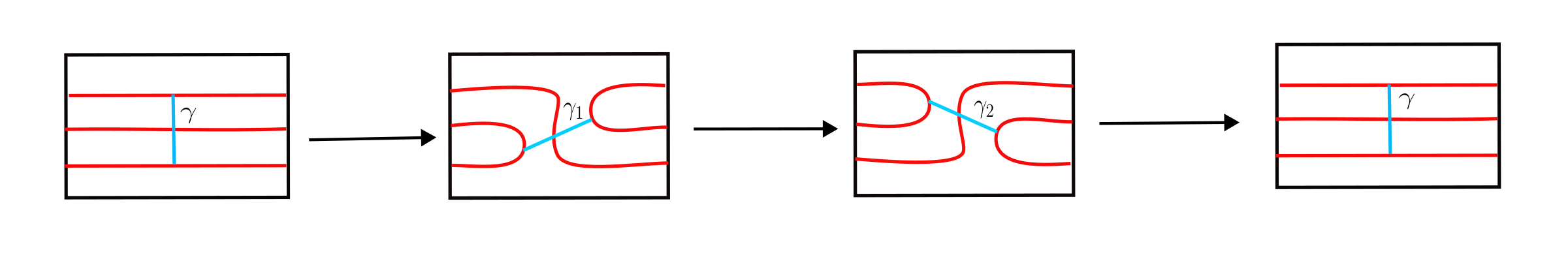}
    \caption{Bypass triangle}
    \label{fig:bypass-triangle}
\end{figure}

We will denote by $\xi_{BT}(\gamma)$ the contact structure on $\Sigma\times I$ obtained after attaching a bypass triangle along $\gamma$ from above at $\Sigma\times\{1\}$. In the cases in which the choice of attaching arc is clear from the context or not relevant we will drop $\gamma$ from the notation.

\subsubsection{The formal class of a bypass triangle}
Consider the 3-contact manifold $(\D^2\times I,\xi_{\inv})$ with $\xi_{\inv}=\ker(fdt+\beta)$, where $t\in I=[0,1]$, 
$f\in C^{\infty}(\D^2)$ and $\beta\in \Omega^1(\D^2)$. This contact structure induces a trivialization of $T(\D^2\times I)$, which defines a constant Gauss map 
$G_{\xi_{\inv}}:\D^2\times I\to \NS^2$.

Consider another plane field $\xi$ on $\D^2\times I$ such that $\xi(p)=\xi_{\inv}(p)$ for all $p \in \Op(\partial(\D^2\times I))$ and the associated Gauss map
$G_{\xi}:\D^2\times I\to \NS^2$. We will denote by $h(\xi)\in \Z$ the \em Hopf invariant \em of the map $G_\xi$. If we denote by $\FC(\D^2\times I,\partial)$ the space of formal contact structures (plane fields) on $\D^2\times I$ that coincide with $\xi_{\inv}$ near the boundary, then there is a bijection $$ \pi_0(\FC(\D^2\times I,\partial))\rightarrow \Z, \xi\mapsto h(\xi); $$ such that $h(\xi_{\inv})=0$. 

A key result that we will use is the following

\begin{theorem}[Huang \cite{Huang-o}]\label{thm:Huang-hopf}
    Consider the contact manifold $(\D^2\times I,\xi_{BT})$ obtained by attaching a bypass triangle from above, then $h(\xi_{BT})=-1$.
\end{theorem}

\subsection{Formal $\F$-embeddings}\label{sec:$F$-emb}

All the contents presented in this section are based on \cite{F-SOT}, to which the reader is referred for a detailed and rigorous exposition of the relevant material.

Let $\Sigma$ be a compact, connected, and oriented surface. Consider a contact manifold $(\Sigma \times \mathbb{R}, \eta)$, where $\eta$ is an $\mathbb{R}$-invariant contact structure.  The surface $\Sigma = \Sigma \times \{0\}$ is convex with respect to $\eta$ and has characteristic foliation $\F$. From now on we will write $\eta=:\xi_{\F}$.

Given a contact 3-manifold $(M, \xi)$, we define the \em space of $\mathcal{F}$-embeddings \em as the set of 
all embeddings of $\Sigma$ into $(M, \xi)$ that induce the characteristic foliation $\mathcal{F}$ on $\Sigma$. This space is denoted by $\Emb(\Sigma, \mathcal{F}, (M, \xi))$. In the case where 
$\partial \Sigma \neq \emptyset$, we impose the additional condition that all embeddings must agree in a neighborhood of the boundary $\partial \Sigma$. 

In order to state the definition of a formal $\mathcal{F}$-embedding, it is necessary to consider germs of formal isocontact embeddings in a neighborhood of the surface. Let $\varepsilon > 0$ be a small positive parameter. We define an abstract \em formal embedding \em from $\Sigma \times (-\varepsilon, \varepsilon)$ into $(M, \xi)$ as a pair $(E, F_s)$ satisfying the following conditions: 
\begin{enumerate}
    \item  $E: \Sigma \times (-\varepsilon, \varepsilon) \to M$ is a smooth embedding, and  
    \item $F_s: T(\Sigma \times (-\varepsilon, \varepsilon)) \to TM$, for $s \in [0,1]$, is a homotopy of bundle isomorphisms covering $E$, such that $F_0 = dE$.
\end{enumerate}
Under these conditions we can define the following equivalence relation:
Let $(E, F_s) : \Sigma\times (-\varepsilon, \varepsilon)\to M$ and $(E', F'_s) : \Sigma\times (-\varepsilon', \varepsilon')\to M$ be two formal embeddings, they are equivalent if there exists 
$0 <\delta < min(\varepsilon, \varepsilon')$ such that
$(E, F_s)|_{\Sigma\times(-\delta,\delta)} = (E',F'_s)|_{\Sigma\times(-\delta,\delta)}$.
Then, we refer to the equivalence class of $(E, F_s)$ under this relation 
as a \em germ of a formal embedding \em 
over $\Sigma$. With a slight abuse of notation, this equivalence class will also be denoted by $(E, F_s)$. The space of germs of formal embeddings is denoted by $\FEmb^{germ}(\Sigma,M)$

Given a contact structure $\xi$ in $\Sigma\times(-\varepsilon,\varepsilon)$, we will say that a formal embedding $(E,F_s)$ is a \emph{formal isocontact embedding} if $F_1(\xi)=\xi|_{E(\Sigma\times(-\varepsilon,\varepsilon))}$.
A formal $\F$-embedding of $\Sigma$ into $(M,\xi)$ is simply a germ of a formal isocontact embedding $(E, F_s) : (\Sigma \times (-\varepsilon, \varepsilon), \xi_{\mathcal{F}}) \to (M, \xi)$. The space of formal $\F$-embeddings is denoted by $\FEmb(\Sigma,\F,(M,\xi))$.
As explained in \cite{F-SOT}, since the space of germs of contact structures inducing the same characteristic foliation over a surface is contractible, the spaces of $\F$-embeddings of surfaces into a given contact $3$-manifold $(M,\xi)$ and germs of isocontact embeddings $(E, dE) : (\Sigma \times (-\varepsilon, \varepsilon), \xi_{\mathcal{F}}) \to (M, \xi)$ are weakly homotopy equivalent. From now on, we will use both interchangeably.

\section{Overtwisted $h$-principles}\label{sec:Overtwisted-h-principles}

We will begin this section by recalling the overtwisted $h$-principle proved by Eliashberg in \cite{EliashbergOT} and several consequences of it. We will then derive the $h$-principle with non-fixed overtwisted disk, that we call loose overtwisted $h$-principle. This is the content of Theorem \ref{teo:LooseOTh-principle} and Corollary \ref{cor:DisksWithOTDisk} and will be crucial in our proof of Theorem \ref{thm:principal}.

\subsection{Overtwisted $h$-principle}

Given a $3$-manifold we will denote by $\FC(M)$ the space of formal contact structures on $M$ and by $\CStr(M)\subseteq \FC(M)$ the subspace of contact structures. In the case that $\partial M$ is non-empty we will assume that all our formal contact structures agree near $\partial M$. Given a subset $G\subseteq M$ equipped with a germ of contact structure we will denote by $\FC(M, \rel G)\subseteq \FC(M)$ the subspace of formal contact structures that induced the given germ over $G$, and we define the corresponding space of contact structures as $\CStr(M,\rel G):=\FC(M,\rel G)\cap \CStr(M)$.

\begin{theorem}[Eliashberg \cite{EliashbergOT}]\label{thm:h-OT}
Let $M$ be a closed $3$-manifold and $\D^2_{\ot} \subset M$ an embedded 2-disk
equipped with the germ of an overtwisted disk. Then, the natural inclusion
$$i : \CStr(M, \rel \D^2_{OT}) \hookrightarrow \FC(M, \rel \D^2_{OT})$$
is a weak homotopy equivalence.
\end{theorem}

For contactomorphisms we have the analogous result as observed by Dymara in \cite{Dymara} for the case of the $3$-sphere and in general by Fern\'andez in \cite{F-SOT} 
\begin{theorem}\label{thm:ContOT}
    Let $(M, \xi)$ be an overtwisted contact 3-manifold and $\Delta_{\ot}\subseteq (M,\xi)$ an overtwisted disk. Then, the inclusion $\Cont(M, \xi,\rel \Delta_{\ot})\hookrightarrow \FCont(M, \xi,\rel \Delta_{\ot})$ is a weak homotopy equivalence.
\end{theorem}

A smooth embedding $e: \D^2 \to (M, \xi)$ is an \emph{overtwisted disk embedding} if the characteristic foliation it induces coincides with that of the overtwisted disk illustrated in Figure \ref{fig:disk-OT}. The \emph{space of overtwisted disk embeddings}, denoted by $\Emb_{\ot}(\mathbb{D}^2, (M, \xi))$, plays a particularly significant role in studying the homotopy type of the contactomorphism group of $(M, \xi)$; see \cite{Dymara, F-SOT, VogelOvertwisted} for further details.

\begin{figure}[h]
    \centering
    \includegraphics[scale=0.3]{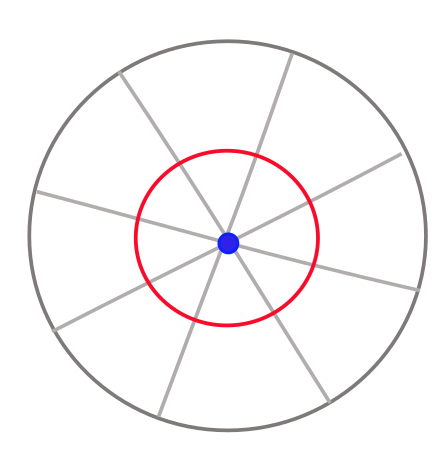}
    \caption{Overtwisted disk. The dividing set is depicted in red and the characteristic foliation in gray.}
    \label{fig:disk-OT}
\end{figure}

Let $\CFra(M, \xi)$ denote the total space of the bundle of contact frames over the contact 3-manifold $(M, \xi)$. The following result can be proved by using Theorem \ref{thm:h-OT}. At the level of path-connected components in $\NS^3$ this was observed by Dymara in \cite{Dymara}. A proof in full generality, that is independent from Theorem \ref{thm:h-OT}, was recently obtained by Farias-Fern\'andez in \cite{F-F}.

\begin{theorem}\label{teo:OvertwistedDisks}
   Let $(M,\xi)$ be an overtwisted contact $3$-manifold and $\Delta_{\ot}\subseteq (M,\xi)$ an overtwisted disk. Then, the $1$-jet evaluation map at the origin 
   $$ \ev_0:\Emb_{\ot}(\D^2,(M\backslash\Delta_{\ot},\xi))\rightarrow \CFra(M\backslash \Delta_{\ot},\xi) $$
   is a weak homotopy equivalence. In particular, two disjoint overtwisted disks are contact isotopic.
\end{theorem}

\subsection{Loose Overtwisted $h$-principle}

In the following discussion we will restrict ourselves to embedding spaces for simplicity. However, the whole discussion trivially applies to germs of embeddings as the ones treated in Section \ref{sec:$F$-emb}. A \em formal embedding \em of a manifold $N$ into a manifold $M$ is a pair $(e,F_s)$ where $e:N\rightarrow M$ is a smooth embedding, and $F_s:TN\rightarrow TM$, $s\in [0,1]$, is a homotopy of bundle monomorphisms covering $e$ such that $F_0=de$. We denote the space of formal embeddings by $\FEmb(N, M)$. We will regard the space of embeddings $\Emb(N,M)$ as a subspace of $\FEmb(N,M)$. Since for a formal embedding $(e, F_s)$, there is no restriction on $F_1$ (besides being a bundle monomorphism) the natural inclusion  
\begin{equation}\label{eq:InclusionFormal}
     i: \Emb(N, M) \hookrightarrow \FEmb(N, M)
\end{equation}  
is a weak homotopy equivalence.

Given a subgroup $G \subseteq \Diff(M)$, there is a natural action of $G$ on $\FEmb(N, M)$ defined via postcomposition, that preserves $\Emb(N,M)$ so it induces an action on this subspace. Specifically, the action is defined as $G \times \FEmb(N, M)\rightarrow \FEmb(N,M), (g, (e, F_s))\mapsto g \cdot (e, F_s) := (g \circ e, dg \circ F_s).$ A subspace $X\subseteq \FEmb(N,M)$ is said to be \em $G$-invariant \em if $G \cdot X = X$.  

\begin{definition}  
Let $(M, \xi)$ be an overtwisted contact $3$-manifold and $K$ a parameter space. A family $(e^k, F_s^k) \in \FEmb(N, M)$, $k \in K$, is said to be a \em loose family \em if there exists a family $d^k \in \Emb_{\ot}(\D^2, (M, \xi))$ such that $e^k(N) \cap d^k(\D^2) = \emptyset$. In this case, we also say that $(e^k, F_s^k)$ is \em loose with respect to \em $d^k$. 
\end{definition}  

Let $(M,\xi)$ be an overtwisted contact $3$-manifold with a fixed overtwisted disk embedding $d \in \Emb_{\ot}(\D^2, (M, \xi))$. For a subspace $X \subseteq \FEmb(N, M)$ we denote by $X_d \subseteq X$ the subspace consisting of embeddings that are loose with respect to $d$.  

Finally, given a subspace $X \subseteq \FEmb(N, M)$, the space of \em loose pairs \em is defined as  
$$\Loose X = \{(e, d) : d \in \Emb_{\ot}(\D^2, (M, \xi)), e \in X_d\}.$$

Notice that if $X\subseteq Y\subseteq \FEmb(N,M)$, then there is a natural inclusion 
$$i : \Loose X \hookrightarrow \Loose Y.$$

\begin{theorem}[Loose Overtwisted $h$-principle]\label{teo:LooseOTh-principle}\hfill\break
 Let $(M,\xi)$ be an overtwisted contact $3$-manifold. Consider two subspaces $X\subseteq Y\subseteq\FEmb(N,M)$ such that
     \begin{enumerate}[label=(\roman*)] 
         \item For every overtwisted disk $d\in \Emb_{\ot}(\D^2,(M,\xi))$ the inclusion 
          $i:X_d\hookrightarrow Y_d$ is a weak homotopy equivalence and
         \item $X$ and $Y$ are $\Cont(M,\xi)$-invariant.
     \end{enumerate}
     Then, the inclusion 
     $$ i:\Loose X\hookrightarrow \Loose Y $$ is a weak homotopy equivalence.
\end{theorem}
\begin{proof}
   
    Applying the Contact Isotopy Extension Theorem, together with condition (ii), we obtain that the projection $\Loose X \to \Emb_{\ot}(\D^2, (M, \xi))$ is a fibration with fiber $X_d$, where $d \in \Emb_{\ot}(\D^2, (M, \xi))$. Similarly, the projection $\Loose Y \to \Emb_{\ot}(\D^2, (M, \xi))$ is a fibration with fiber $Y_d$.
    
    Therefore, in the following commutative diagram the rows are fibrations
    \begin{displaymath} 
    \xymatrix@M=10pt{
      X_d \ar@{^{(}->}[r] \ar@{^{(}->}[d] & \Loose X \ar[r] \ar@{^{(}->}[d] & \Emb_{\ot}(\D^2,(M,\xi) ) \ar@{^{(}->}[d] \\
     Y_d \ar@{^{(}->}[r] & \Loose Y \ar[r] &  \Emb_{\ot}(\D^2,(M,\xi)) }
    \end{displaymath} 
    
    The inclusion between the fibers is a homotopy equivalence due to condition (i) of the hypothesis, while the inclusion between the bases is simply the identity map. In particular, the latter is also a homotopy equivalence.  Therefore, by the Five Lemma on the long exact sequence of homotopy groups, it follows that the inclusion  
    $\Loose X \hookrightarrow \Loose Y$ is a weak homotopy equivalence.
\end{proof}

\subsection{Convex disks with overtwisted disks in the complement}

We will apply Theorem \ref{teo:LooseOTh-principle} to the space of embeddings of convex disks with Legendrian boundary and fixed characteristic foliation $\F$ into an overtwisted contact $3$-manifold $(M,\xi)$ and its formal counterpart. Namely, the spaces 
$$ \Emb(\D^2,\F,(M,\xi),\rel \partial)\subseteq \FEmb(\D^2,\F,(M,\xi),\rel \partial).$$

The following holds

\begin{corollary}\label{cor:DisksWithOTDisk}
    The inclusion 
    $$ \Loose \Emb(\D^2,\F,(M,\xi),\rel \partial) \hookrightarrow \Loose \FEmb(\D^2,\F,(M,\xi),\rel \partial) $$
    is a weak homotopy equivalence.
\end{corollary}
\begin{proof}
    Note that $\Emb(\D^2,\F,(M,\xi),\rel \partial)$ and $\FEmb(\D^2,\F,(M,\xi),\rel \partial)$ are subspaces of the space of germs of formal embeddings $\FEmb^{germ}(\D^2,M,\rel \partial)$, see Section \ref{sec:$F$-emb}. Let $\Lambda\subseteq (M,\xi)$ be the Legendrian boundary of the disks. Our goal is to apply Theorem \ref{teo:LooseOTh-principle} relative to a small open set $\Op(\Lambda)$, where $X=\Emb(\D^2,\F,(M,\xi),\rel \partial)$ and $Y=\FEmb(\D^2,\F,(M,\xi),\rel \partial)$.

    Clearly $X$ and $Y$ are $\Cont(M,\xi,\rel \Lambda)$-invariant. Therefore, it remains to check condition (i) in Theorem \ref{teo:LooseOTh-principle}. That is, that for every overtwisted disk embedding $d\in \Emb_{\ot}(\D^2,M)$ the inclusion $X_d\hookrightarrow Y_d$ is a weak homotopy equivalence. The proof of the latter is analogous to the one given in \cite[Proposition 3.6]{F-SOT} and the details are left to the reader. 
\end{proof}

\section{Creating isotopies of overtwisted disks}\label{sec:Creating-isotopies-of-overtwisted-disks}
Let $(M,\xi)$ be a closed overtwisted contact $3$-manifold. We will consider an isotopy of smooth disks $e^t:\D^2\hookrightarrow M$, $t\in [0,1]$, all of them coinciding near the Legendrian boundary, that we assume to be discretized. That is, there exists a finite sequence of times $ 0 = t_0 < t_1 < \ldots < t_n = 1 $ such that the image of 
$$ E_{i} : \D^2 \times [t_i, t_{i+1}]  \to (M, \xi), \quad (p,t)\mapsto e^t(p), $$ is an embedded pinched $3$-disk and $E_{i}$ is a smooth embedding away from $\Op(\partial \D^2)\times[t_i,t_{i+1}]$. We define $D_i:=\image(E_i)$. We will say that $D_i$ and $E_i$ are going upwards, resp. downwards, if $E_i$ is orientation preserving, resp. orientation reversing. A pinched $3$-disk going upwards is illustrated in Figure \ref{fig:pinched-disk}. We will further assume that $e^{t_i}$ is convex for every $i\in\{0,\ldots,n\}$ with characteristic foliation $\mathcal{F}_i$. 
\begin{figure}[h]
    \centering
    \includegraphics[scale=0.3]{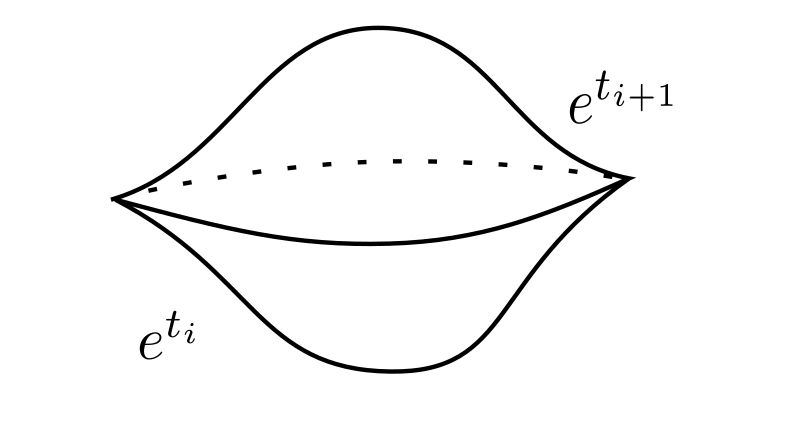}
    \caption{A pinched 3-disk}
    \label{fig:pinched-disk}
\end{figure}

The goal of this Section is to find an isotopy of overtwisted disks $\Delta^t$, $t\in[0,1]$, such that $\Delta^t\cap \image(e^t)=\emptyset$ for all $t\in [0,1]$. This will be key to apply Corollary \ref{cor:DisksWithOTDisk} to prove Theorem \ref{thm:principal} latter. We divide the argument into two cases, one in which the disks $(D_i,\xi)$ are tight and other in which they are overtwisted. We need also to interpolate between both cases, therefore we need to find isotopies of overtwisted disks with prescribed initial conditions. 

\subsection{Tight case}
First, we deal with the tight case:
    
\begin{proposition}[Tight case] 
  \label{prop:TightIsotopy}
        Let $\Delta_{\ot}$ be an overtwisted disk in the complement of $e^0$. Assume that the pinched $3$-disks $(D_i,\xi)$ are tight for every $i\in\{0,\ldots,n\}$. Then, there exists an isotopy $\Delta^t$, $t\in [0,1]$, of overtwisted disks such that 
        \begin{itemize}
            \item $\Delta^0=\Delta_{\ot}$ and 
            \item $\Delta^t\cap \image(e^t)=\emptyset$ for all $t\in[0,1]$.
        \end{itemize}
\end{proposition}
\begin{proof}
Consider a $1$-parameter family of smooth $3$-disks $B^t=\Op(\image(e^t))$, $t\in[0,1]$. We take $B^0$ small enough so that $B^0\cap \Delta_{\ot}=\emptyset$. By the hypothesis we can assume that $(B^t,\xi)$ is tight for every $t\in[0,1]$. In particular, the $2$-spheres $S^t:=\partial B^t$ are tight for every $t\in [0,1]$. Write $\Lambda=e^t(\partial \D^2)$ and notice that $S^t \cap \image(e^t)=\emptyset$ for every $t\in [0,1]$ and in particular $S^t\cap \Lambda =\emptyset$.

By \cite{FMP20}, or a small adaptation of \cite{Colin}, we can ensure the existence of a $1$-parameter family of convex spheres with the same characteristic foliation $\hat{S}^t$, $t\in [0,1]$, which is $C^0$-close to $S^t$. We can also achieve that $\hat{S}^0\cap \Delta_{\ot}=\emptyset$, $\hat{S}^t \cap \image(e^t)=\emptyset$  and $\hat{S}^t\cap \Lambda =\emptyset$ for every $t\in [0,1]$.

To conclude apply the Contact Isotopy extension Theorem to find a contact isotopy $\varphi^t\in \Cont(M,\xi)$, $t\in [0,1]$, such that $\varphi^0=\Id$, $\varphi^t(\hat{S}^0)=\hat{S}^t$ and $\varphi^t_{|\Op(\Lambda)}=\Id_{|\Op(\Lambda)}$. The isotopy of overtwisted disks is then given by $\Delta^t:=\varphi^t(\Delta_{\ot})$. 
\end{proof}

\subsection{Overtwisted case} 
Finally we address the overtwisted case. In order to do so we need some preparatory lemmas. The first one addresses the problem without prescribing the initial condition. 

\begin{lemma}\label{lema:caracol}
        Assume that the pinched $3$-disks $(D_i,\xi)$ are overtwisted for every $i\in \{0,\dots,n\}$. If $n>1$ further assume that the embedding $e^{t_i}$ is overtwisted for every $i\in \{1,\ldots, n-1\}$. Then, there exists an isotopy of overtwisted disks $\Delta^t$,$t\in [0,1]$, such that $\image(e^t)\cap \Delta^t=\emptyset$.
\end{lemma}
\begin{proof}
We will consider several overtwisted disks near each disk $e^{t_i}$, $i\in\{1,\ldots, n-1\}$, and explain how to interpolate appropriately between them. We will do this inductively over the overtwisted pinched $3$-disks $(D_i,\xi)$. 

Since $e^{t_i}$ is convex for every $i\in \{0,1,\ldots,n\}$ and, by hypothesis, every small neighborhood $(\Op(\image(e^{t_i})),\xi)$ is overtwisted for every $i\in \{1,\ldots,n-1\}$, we can find a sufficiently small $\varepsilon>0$ and overtwisted disks $\Delta^{i,\pm}_{j}$, for $(i,j)\in \{1,\ldots,n-1\}\times \{1,3\}$, such that 
\begin{itemize}
    \item [(a)] $\Delta^{i,\pm}_{j}\subseteq \Op(\image(e^{t_i\pm j\varepsilon}))$. In particular, $\Delta^{i,-}_{j}\subseteq D_{i-1}$ and $\Delta^{i,+}_{j}\subseteq D_{i}.$
    \item [(b)] $\Delta^{i,\pm}_{j}\cap \image(e^{t_i})=\emptyset$.
    \item [(c)] $\Delta^{i,\pm}_{j}\cap \image(e^{t_i\pm 2 \varepsilon})=\emptyset$.
    \item [(d)] $\Delta^{i,+}_{1},\Delta^{i,+}_{3},\Delta^{i,-}_{1}$ and $\Delta^{i,-}_{3}$ are pairwise disjoint for every $i\in\{1,\ldots,n-1\}$.
    \item [(e)] $\Delta^{i,+}_{1},\Delta^{i,+}_{3},\Delta^{i+1,-}_{1}$ and $\Delta^{i+1,-}_{3}$ are pairwise disjoint for every $i\in \{1,\ldots,n-2\}$.
\end{itemize}
The way of finding these overtwisted disks is straightforward: we can assume that there exists a positive $\varepsilon>0$ such that the isotopies $e^{t},t\in [t_i-4\varepsilon,t_i+4\varepsilon]$, are composed of disks with fixed characteristic foliation; hence we can
consider an $I$-invariant neighborhood of $e^{t_i}$ and
an overtwisted disk 
in a $\sigma$--small neighborhood of $e^{t_i}$ with $\sigma\ll\varepsilon$ and several parallel copies of it. The last property can be ensured since all those overtwisted disks lie in the same pinched $3$-disk $D_i$. 

Now we explain how to define the required isotopy of overtwisted disks. We proceed inductively over the pinched $3$-disks $D_i$, $i\in \{0,\ldots,n-1\}$. We start with $D_0$, that is, for $t\in [0,t_1]$. 

Observe that $\Delta^{1,-}_{1}$ and $\Delta^{1,-}_{3}$ are disjoint by (d) above. Moreover, they are disjoint from $\image(e^{t_1-2\varepsilon})$ by (c). Hence, since two disjoint overtwisted disks in the connected overtwisted contact $3$-manifold $M\backslash \image(e^{t_1-2\varepsilon})$ are contact isotopic, by Theorem \ref{teo:OvertwistedDisks} we can find an isotopy of overtwisted disks $O_t$, $t\in [0,1]$, and $\delta>0$ very small such that $O_0=\Delta^{1,-}_{1}$, $O_1=\Delta^{1,-}_{3}$ and $O_t\cap \image(e^{t_1-2\varepsilon + s})=\emptyset$ for $s\in[-\delta, \delta]$.  See Figure \ref{fig:case3}.
\begin{figure}[h]
    \centering
    \includegraphics[scale=0.35]{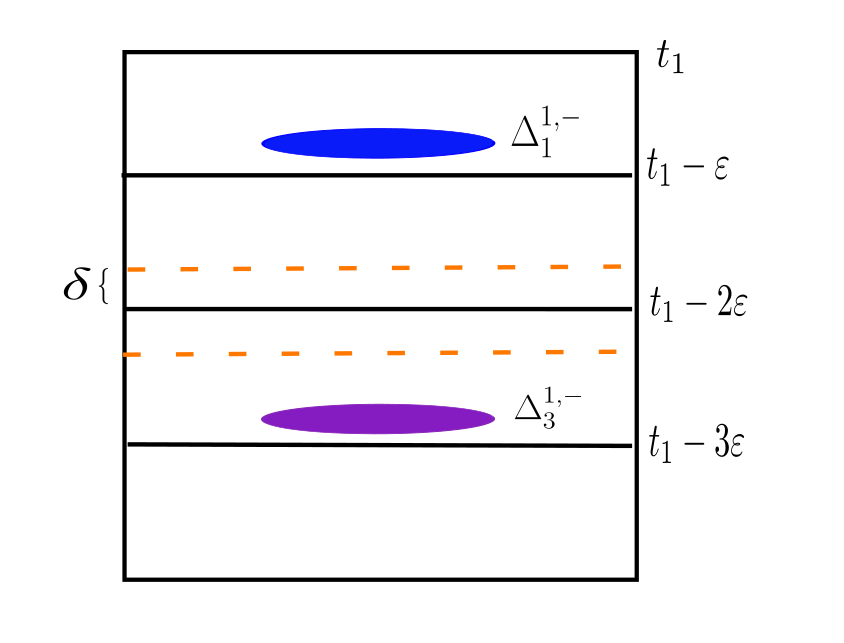}
    \caption{Schematic of the construction over $D_0$.}
    \label{fig:case3}
\end{figure}

Therefore, for $ t \in [0, t_1] $, we can define:
$$\Delta^t= \begin{cases} \Delta^{1,-}_{1} \quad t\in[0,t_1-2\varepsilon], \\
O_{\frac{t - (t_1 - 2\varepsilon)}{\delta}} \quad t\in[t_1-2\varepsilon,t_1-2\varepsilon+\delta],\\
\Delta^{1,-}_{3} \quad t\in[t_1-2\varepsilon+\delta, t_1].\end{cases}$$

This defines the required isotopy over $D_0$. To extend it over $D_1$, that is for $t\in[t_1,t_2]$, proceed as follows. First for $t\in [t_1,t_1+\delta]$ interpolate between $\Delta^{1,-}_{3}$ and $\Delta^{1,+}_{3}$. This can be done because of property (b) and (d) above. Secondly, for $t\in [t_1+\delta,t_1+2\delta]$ interpolate between $\Delta^{1,+}_{3}$ and $\Delta^{2,-}_{1}$. This can be done because of property (e). In both cases we are using Theorem \ref{teo:OvertwistedDisks}. Finally, proceed exactly as before to define the isotopy over $D_1$. Continue this process inductively to conclude. 
\end{proof}

The following key lemma will allow us to assume that if $\F_0$ and $\F_n$ are tight characteristic foliations then we can assume that both coincide after possibly adding several \em tight \em bypasses, which causes no problem due to Proposition \ref{prop:TightIsotopy}.

\begin{lemma}\label{lema:FixingDividingSet}
      Let $(\D^2_i, \Gamma_i)$, $i\in\{0,1\}$, be two convex disks with the same Legendrian boundary with tight dividing sets $\Gamma_i$ such that 
     \begin{enumerate}[label=(\roman*)]
          \item $\Gamma_0$ and $\Gamma_1$ coincide near the boundary.
          \item $-\frac{1}{2}\#(\partial\D^2_1\cap \Gamma_1)=-\frac{1}{2}\#(\partial\D^2_0\cap \Gamma_0)$ and 
          \item $\chi(\D^2_{1,+})-\chi(\D^2_{1,-})=\chi(\D^2_{0,+})-\chi(\D^2_{0,-}).$
      \end{enumerate}
      Then, there exists a sequence of tight abstract bypasses such that after attaching it to $\D^2_0$ from above turns $\Gamma_0$ into $\Gamma_1$, up to isotopy. There also exists a sequence of tight abstract bypasses turning $\Gamma_0$ into $\Gamma_1$ when attached to $\D^2_0$ from below.
\end{lemma}
\begin{remark}
    This result is likely well-known, although we were unable to find a direct reference. The existence of a sequence of bypasses from both above and below, which transforms one dividing set into the other, follows immediately, after Theorem \ref{thm:Kanda}, from the Eliashberg-Fraser classification of Legendrian unknots \cite{Eliashberg-Fraser}, and the Colin-Honda discretized isotopy (Theorem \ref{thm:Discretization}) between convex surfaces with Legendrian boundary. This is based on the observation that each such disk contactly embeds into $(\mathbb{R}^3, \xi_{\std})$, as a consequence of Eliashberg's classification of tight contact structures on $\mathbb{R}^3$ \cite{EliashbergTight}.
\end{remark}
\begin{proof}
    We prove the existence of the sequence of bypasses attached from above. The other case is analogous. We proceed by induction on the number of curves in the dividing set. If there is only one curve, the disk is standard and there is nothing to prove. 
    
    Now, assume that there are $n$ curves in the dividing set, with $n > 1$. If all the curves are $\partial$-parallel (the disk is multi-standard in the terminology of \cite{FMP22}), there is nothing to prove, since there is only one possible configuration. Otherwise, let $\mathcal{P} = \{ p_1, \ldots, p_{2k} \}$ denote the set of points ordered cyclically and counterclockwise such that $\partial \D_1^2 \cap \Gamma_1 = \partial \D_0^2 \cap \Gamma_0=\mathcal{P}$. Since $n > 1$, there exists a component $\Gamma_{1,1}$ of $\Gamma_1$ that intersects the boundary into two consecutive points of $\mathcal{P}$, i.e. $\Gamma_{1,1}$ is $\partial$-parallel. Without loss of generality, suppose that these points are $p_1$ and $p_2$.
    
    Let $\Gamma_{0,0}$ denote the component of the dividing set $\Gamma_0$ that intersects the point $p_1$. This component must also intersect another point of the form $p_{2 + 2j}$, where $j \in \mathbb{Z}$. If $j=0$ there is nothing to do. There are other two cases to consider: 
    
    \begin{figure}[h]
    \centering
    \includegraphics[scale=0.3]{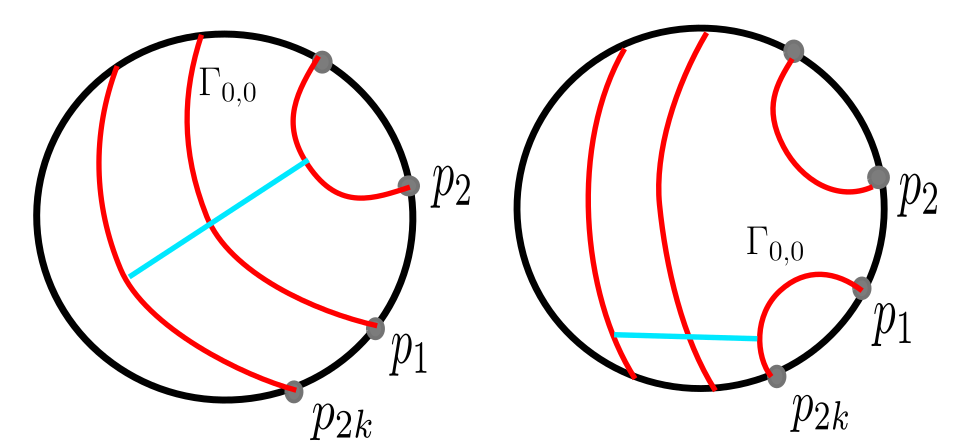}
    \caption{Possible configurations of dividing sets described in the proof. At the left it is depicted the case $p_{2 + 2j} \neq p_{2k}$ and at the right the case $p_{2 + 2j} = p_{2k}$}
    \label{fig:DividingSets.}
\end{figure}
   
    First, if $p_{2 + 2j} \neq p_{2k}$, then we consider the two components of $\Gamma_0$ that contain $p_{2k}$ and $p_2$. Notice that they must be different. Attach from above the obvious bypass whose attaching arc intersects these two components and $\Gamma_{0,0}$. After this, $\Gamma_{0,0}$ is transformed into a component that connects $p_1$ to $p_2$.

    Second, if $p_{2 + 2j} = p_{2k}$. By the hypothesis, there exists a component $\Gamma_{0,2}$ that is not $\partial$-parallel and that can be connected with $\Gamma_{0,0}$ by an embedded arc whose interior does not intersect the dividing set. We can then find another component $\Gamma_{0,k}$ in the half-disk delimited by $\Gamma_{0,2}$ that does not contain $\Gamma_{0,0}$, for which the previous arc can be extended to intersect this component in such a way that it only intersects the dividing curves at three points. Attach a bypass, with attaching arc at this previous arc. After this, $\Gamma_{0,0}$ no longer connects $p_1$ to $p_{2k}$, as in the previous case.
    
    Hence, after attaching one or two tight bypasses from above we can arrange that both dividing sets coincide on the curve connecting $p_1$ to $p_2$. Since $n>1$ we can apply Theorem \ref{teo:realizacion-Giroux} (or the Legendrian Realization Principle from \cite{Honda}) to find a parallel Legendrian arc to $\Gamma_{0,0}$, which delimits a subdisk with Legendrian boundary. This subdisk has only one component of the dividing set, namely $\Gamma_{0,0}$. Therefore, the other subdisk associated with this Legendrian curve has $n - 1 < n$ components in both dividing sets. Finally, we can use the inductive hypothesis to conclude.
\end{proof}

Now we are ready to address the case in which the pinched $3$-disks are overtwisted. In this case, we cannot find the isotopy of overtwisted disks for the given smooth isotopy $e^t$ and a priori given starting overtwisted disk, in general, but we will be able to deform our isotopy $e^t$ into a new one for which we can solve the problem. That is the content of the following

\begin{proposition}[Overtwisted case]\label{prop:OvertwistedIsotopy}
    Let $\Delta_{\ot}$ be an overtwisted disk in the complement of $e^0$. Assume that the pinched $3$-disks $(D_i,\xi)$ are overtwisted for $i\in \{0,\ldots,n-1\}$, and that the embeddings $e^0$ and $e^1$ are tight. If $n>1$ further assume that the embedding $e^{t_i}$ is overtwisted for $i\in\{1,\ldots,n-1\}$. Then, there exists a homotopy of isotopies $e^{t,u}:\D^2\hookrightarrow M$, $(t,u)\in[0,1]^2$, relative to the boundary, such that
    \begin{enumerate}[label=(\roman*)]
        \item $e^{t,u}=e^t$ for all $(t,u)\in [0,1]\times \{0\}\cup \{0,1\}\times[0,1]$,
        \item The characteristic foliation of $e^{t,1}$ is constant for $0\leq t\leq 1/2$,
        \item $e^{t,1}$, $t\in[\frac{1}{2},1]$, is a discretized isotopy composed of tight pinched $3$-disks; and
        \item There exists an isotopy of overtwisted disks $\Delta^t$, $t\in[0,1]$, such that $\Delta^0=\Delta_{\ot}$ and $\Delta^t\cap \image(e^{t,1})=\emptyset$.
    \end{enumerate}
\end{proposition}    

\begin{proof}
    Consider the overtwisted pinched $3$-disk $(D_{n-1},\xi)$. If this pinched $3$-disk is going downwards (resp. upwards) we can use Lemma \ref{lema:FixingDividingSet} to find a sequence of tight bypasses that after attaching them to $e^1$ from above (resp. below) turns the dividing set of $e^1$ into the one of $e^0$. Notice that we can apply this Lemma because of Theorem \ref{thm:Kanda}. Since the pinched $3$-disk $(D_{n-1},\xi)$ is overtwisted we can realize such a sequence inside this pinched $3$-disk by Lemma \ref{lem:bypassAttachingArc}. Moreover, by Giroux Realization Theorem \ref{teo:realizacion-Giroux} we may further assume that the disk obtained from $e^1$ after attaching such a sequence of bypasses has the same characteristic foliation as $e^0$. 
    
    In conclusion, after a possible smooth isotopy supported in the pinched $3$-disk $(D_{n-1},\xi)$ we may assume that there exists a time $\hat{t}\in (t_{n-1}, 1]$ such that $e^{\hat{t}}$ is convex with the same characteristic foliation than $e^0$ and such that the isotopy $e^t$, $t\in[\hat{t},1]$, decomposes as a sequence of tight pinched $3$-disks.
    
    From now on, we will focus on the isotopy $e^t$, $t\in [0,\hat{t}]$. Denote by $\F:=\F_0$ the characteristic foliation of $e^0$ and $e^{\hat{t}}$. Our goal is to use the $h$-principle for $\F$-embeddings, Corollary \ref{cor:DisksWithOTDisk}, with this isotopy. However, there may be a formal obstruction to equip the smooth isotopy $e^t$, $t\in [0,\hat{t}]$, with a formal structure and apply the $h$-principle. In the next lines we will explain how to compute such a obstruction and how to sort it out by possibly adding more tight pinched $3$-disks.  

 To compute the formal obstruction we proceed as follows. First, use use the smooth isotopy $e^t$, $t\in[0,\hat{t}]$, to extend the formal $\F$-embedding data of $e_0$ (which is actually genuine) up to $e^{\hat{t}}$; thereby obtaining a formal isotopy $(e^t, F^{t}_{s})$, $(t,s)\in[0,\hat{t}]\times [0,1]$. We need to describe the homotopy obstruction to making $F^{\hat{t}}_s$ constant. Define a plane field $\eta=\eta(p, \tau)$ in $\D^2\times[0,1]$ as follows:
 $$ \eta (p,\tau)=F^{\hat{t}}_{\tau}(\xi(p)). $$

This plane field coincides with the $I$-invariant contact structure $(\D^2\times [0,1],\xi_{\inv})$ defined by the convex characteristic foliation $\F$ over $\partial (\D^2\times[0,1])$. The obstruction to making $F^{\hat{t}}_s$ constant is precisely the obstruction to homotoping the plane field $\eta$ into $\xi_{\inv}$, relative to the boundary. To analyze this, we trivialize the tangent bundle of $\D^2\times[0,1]$ using $\xi_{\inv}$, in such a way that $\xi_{\inv}$ has constant Gauss map, which allows us to identify this obstruction with the Hopf invariant of the associated Gauss map $G_{\eta}:(\D^2\times[0,1],\partial)\rightarrow (\NS^2,p_N)$. We denote this invariant by $h(\eta)\in\Z$. 
    
If $h(\eta)=0$ then there is nothing to do to obtain this formal structure. Let us assume that $h(\eta)\neq 0$. We suppose that obvious pinched $3$-disk $\hat{D}$ bounded by $e^{t_{n-1}}$ and $e^{\hat{t}}$ is going upwards, the other case is analogous.

    \underline{If $h(\eta)<0$:} Since the pinched $3$-disk $(\hat{D},\xi)$ is overtwisted near $\image(e^{t_{n-1}})$, we may add $|h(\eta)|$ bypass triangles \em downwards \em to $e^{\hat{t}}$ (Lemma \ref{lem:bypassAttachingArc}). This allows us to find a new time $t^*\in(t^{n-1},\hat{t})$ such that, after a possible isotopy, $e^{t^*}$ is an $\F$-embedding. The obvious pinched $3$-disk between $e^{t^*}$ and $e^{\hat{t}}$ is given by a sequence of bypass triangles, in particular, it decomposes as a sequence of tight pinched $3$-disks. Finally, adding a bypass triangle downwards increases the Hopf invariant by $1$ as stated in Theorem \ref{thm:Huang-hopf}. Therefore, there is no obstruction to equip the isotopy $e^t$, $t\in[0,t^*],$ with a formal structure.  

    \underline{If $h(\eta)>0$:} We proceed similarly to the previous case. First, create a bypass triangle downwards to obtain a new $\F$-embedding $e^{\tilde{t}}$, for some $\tilde{t}\in (t_{n-1},\hat{t})$. Note that by Theorem \ref{thm:Huang-hopf} the obstruction to equip the isotopy $e^t$, $t\in[0,\tilde{t}]$, is now $h(\eta)+1$. However, the pinched $3$-disk enclosed by $e^{\tilde{t}}$ and $e^{\hat{t}}$ is overtwisted, since it is described as an embedded bypass triangle. Therefore, we may find a total of $h(G_{\eta})+1$ bypass triangles attached to $e^{\tilde{t}}$ from above inside this pinched $3$-disk. After this, it is possible to find a time $t^*\in (\tilde{t},\hat{t})$ such that, after a possible isotopy,  $e^{t^*}$ is an $\F$-embedding and the isotopy $e^t$, $t\in[0,t^*]$, admits a formal structure. Moreover, the isotopy $e^t$, $t\in [t^*, \hat{t}],$ is isotopic, relative to the boundary, to the isotopy that first interpolates back from $e^{t^*}$ to $e^{\tilde{t}}$, following the sequence of bypass triangles backwards, and then realizes obvious isotopy from $e^{\tilde{t}}$ to $e^{\hat{t}}$ given by the initial bypass triangle upwards. In particular, we may assume that it decomposes as a sequence of tight pinched $3$-disks. 

    In all the cases, we found an $\F$-embedding $e^{t^*}$, with $t^*\in(t_{n-1},\hat{t}]$, such that 
    \begin{itemize}
        \item The smooth isotopy $e^t$, $t\in[0,t^*]$, admits a formal structure and is discretized by overtwisted pinched $3$-disks.
        \item The isotopy between $e^t$, $t\in[t^*,1]$ is discretized by a sequence of \em tight \em pinched $3$-disks. 
    \end{itemize}

    We are now ready to conclude the proof. First, we apply Lemma \ref{lema:caracol} to construct an isotopy of overtwisted disks $\Delta^t$, $t\in[0,t^*],$ such that $ \image(e^t)\cap \Delta^t=\emptyset $ for all $t\in[0,t^*]$. We apply the
    $h$-principle given in Corollary \ref{cor:DisksWithOTDisk} to $(e^t,\Delta^t)$, $t\in[0,t^*]$, and perform an obvious reparametrization of the $t$-variable, to conclude the existence of a homotopy $e^{t,u}$, $(t,u)\in[0,1]^2$, satisfying properties (i), (ii) and (iii) stated above. To achieve property (iv), we simply use the Contact Isotopy Extension theorem for $t\in[0,\frac{1}{2}]$ in combination with Proposition \ref{prop:TightIsotopy} for $t\in[\frac{1}{2},1]$. This concludes the argument.
    \end{proof}

\subsection{The general case}
    We conclude this Section with the general existence statement of the isotopy of overtwisted disks that we will need. 

    \begin{theorem}\label{thm:IsotopyOvertwistedDisks}
       Let $(M,\xi)$ be a closed overtwisted contact $3$-manifold. Let $e^t:\D^2\hookrightarrow M$, $t\in[0,1]$, be an isotopy of smooth embeddings fixed near the Legendrian boundary $\Lambda=e^t(\partial \D^2)$ between two convex disks $e^0$ and $e^1$ with the same characteristic foliation $\F$. Assume that either $\F$ is tight or $\tb(\Lambda)+|\rot(\Lambda)|>-1$. Then, there exists a homotopy of isotopies $e^{t,u}:\D^2\hookrightarrow M$, $(t,u)\in [0,1]^2$, relative to the boundary, and an isotopy of overtwisted disks $\Delta^t$, $t\in [0,1]$, such that 
       \begin{itemize}
       \item [(i)] $e^{t,u}=e^t$ for $(t,u)\in [0,1]\times\{0\}\cup \{0,1\}\times[0,1]$
       \item [(ii)] $\Delta^t\cap \image(e^{t,1})=\emptyset$.
       \end{itemize}
    \end{theorem}
    \begin{proof}
       We may assume that the isotopy is discretized as at the beginning of the section. If $\tb(\Lambda)+|\rot(\Lambda)|>-1$, where the rotation is measured with respect to one the disks of the given isotopy, then, by the Bennequin-Eliashberg inequality (Theorem \ref{thm:Bennequin-Eliashberg}), $(\Op(\image(e^t)),\xi)$ is overtwisted for all $t\in [0,1]$. Hence, we can directly apply Lemma \ref{lema:caracol} to find the required isotopy of overtwisted disks $\Delta^t$, $t\in [0,1]$, and simply set $e^{t,u}\equiv e^t$.

       Assume now that $\F$ is tight, and in particular $\tb(\Lambda)+|\rot(\Lambda)|\leq-1$. Since $\F$ is tight then the discretized isotopy decomposes in smaller discretized isotopies as the ones treated in Propositions \ref{prop:TightIsotopy} and \ref{prop:OvertwistedIsotopy}. By Colin's Theorem \ref{Colin97} we may fix an overtwisted disk $\Delta_{\ot}$ in the complement of the tight disk $e^0$. Declare $\Delta^0=\Delta_{\ot}$ and apply Propositions \ref{prop:TightIsotopy} and \ref{prop:OvertwistedIsotopy} iteratively to conclude. 
    \end{proof}

    \begin{remark}\label{rmk:OTDisks}
        This Theorem is the key result that does not hold, in general, in the case in which $\tb(\Lambda)+|\rot(\Lambda)|\leq -1$ and $\F$ is overtwisted. In this case, it is not true that every discretized isotopy decomposes in smaller discretized isotopies as in Propositions \ref{prop:TightIsotopy} and \ref{prop:OvertwistedIsotopy}. The case that could appear and is not covered by these results is when there is a discretized isotopy $e^t:\D^2\rightarrow M$, $t\in [0,1]$, composed of two overtwisted pinched $3$-disks $(D_0,\xi)$ and $(D_1,\xi)$ for which $D_0\cap D_1$ has a tight neighborhood while the embeddings $e_0$ and $e_1$ are overtwisted. In this case one cannot, in general, interpolate between the overtwisted disks for $(D_0,\xi)$ and the ones for $(D_1,\xi)$ since they could intersect. However, under the stronger assumption that the space of overtwisted disks is connected, we could also sort out this case. Notice that this assumption does not always hold \cite{VogelOvertwisted} but is true for a wide class of overtwisted contact $3$-manifolds \cite{F-SOT}.
    \end{remark}

\section{Proof of Theorem \ref{thm:principal}}\label{sec:ProofMainTheorem}

We start by addressing the surjectivity statement in Theorem \ref{thm:principal}:
  
 \begin{lemma}[Surjectivity] \label{lema:pi0-sobre} \hfill\break
     Given $(e,F_s)\in \FEmb(\D^2,\F,(M,\xi),\rel\partial)$ there is a homotopy $(e^t,F_s^t)\in\FEmb(\D^2,\F,(M,\xi),\rel \partial)$, with $t\in[0,1]$ such that
    \begin{enumerate}[label=(\roman*)]
        \item $(e^0,F_s^0)=(e,F_s)$,
        \item $(e^1,F_s^1)=(e^1,de^1)\in\Emb(\D^2,\F,(M,\xi),\rel\partial)$, for all $s\in[0,1]$.
        \end{enumerate}
\end{lemma}
\begin{proof} Let $D=e(\D^2\times\{0\})$ and $\gamma =\partial D$. By hypothesis, we have $\tb(\gamma) \leq 0$. Using Proposition \ref{teo:Giroux-Honda}, we can assume that $D$ is convex. We can ensure the existence of an overtwisted disk $\Delta_{\ot}$ in the complement of a small neighborhood of $D$. Indeed, let $\Gamma$ be a dividing set of $D$. Then, if $\Gamma$ contains a closed curve, by Giroux's Criterion (Theorem \ref{teo:Criterio-Giroux}), an $I$-invariant neighborhood of $D$ is overtwisted and there exists an overtwisted disk in the complement of $D$. On the other hand if $\Gamma$ does not contain closed curves, by Giroux's Criterion (Theorem \ref{teo:Criterio-Giroux}) we obtain that $(e(\D^2 \times I_\varepsilon),\xi)$ is tight. Consequently, by Colin's Theorem \ref{Colin97}, $M \setminus e(\D^2 \times I_\varepsilon)$ is overtwisted. In both cases, we can assume the existence of an overtwisted disk $\Delta_{\ot}\subseteq M \setminus e(\D^2 \times I_\varepsilon)$. Now the proof follows from Corollary \ref{cor:DisksWithOTDisk}.
\end{proof}

 It remains to verify the injectivity of the Theorem \ref{thm:principal}, which reduces to proving the following:
 
\begin{lemma}[Injectivity]
    $$\pi_1 (\FEmb(\D^2,\F,(M,\xi),\rel \partial),\Emb(\D^2,\F,(M,\xi),\rel \partial))=0.$$
\end{lemma}
\begin{proof}
Write $X=\Emb(\D^2,\F,(M,\xi),\rel \partial)$ and $Y=\FEmb(\D^2,\F,(M,\xi),\rel \partial)$. Theorem \ref{thm:IsotopyOvertwistedDisks} implies that the natural map 
$$ \pi_1(\Loose Y, \Loose X)\rightarrow \pi_1(Y,X)  $$ is surjective. Since $\pi_1(\Loose Y, \Loose X)=0$ by Corollary \ref{cor:DisksWithOTDisk} the result follows.
\end{proof}

\section{Applications}\label{sec:Applications}

The purpose of this Section is to prove Corollaries \ref{cor:unknots}, \ref{cor:UnknotsS3} and \ref{cor:pi0complementounknot}. We will first prove and recall some results for overtwisted contact $3$-spheres that we will need for Corollaries \ref{cor:UnknotsS3} and \ref{cor:pi0complementounknot}. Finally, we will prove the mentioned Corollaries.

\subsection{Convex disks with Legendrian boundary in an overtwisted $3$-sphere}

We begin by recalling the following

\begin{theorem}[Hatcher \cite{Hatcher-irreducible,Hatcher}]
\label{thm:HatcherDisks}
The space $\Emb(\D^2,\NS^3,\rel \partial)$ is contractible.
\end{theorem}

This result together with Theorem \ref{thm:principal} allow us to conclude the following

\begin{corollary}\label{cor:DisksS3}
    Let $(\NS^3,\xi)$ be an overtwisted contact $3$-sphere and $\D^2\subseteq (\NS^3,\xi)$ an embedded convex disk with Legendrian boundary and characteristic foliation $\mathcal{F}$. Assume that either $\F$ is tight or $\tb(\partial\D^2)+|\rot(\partial \D^2)|>-1$. Then, $$ \pi_0( \Emb(\D^2,\F,(\NS^3,\xi),\rel \partial))\cong \Z. $$
\end{corollary}
\begin{proof}
   
By Theorem \ref{thm:principal} it is enough to check that $\pi_0( \FEmb(\D^2,(\NS^3,\xi),\rel \partial))\cong \Z.$ Consider the following fibration:
    \begin{equation}\label{fibracion-Path}
        \Path_{(*, \U(1))}\operatorname{Iso}(T\NS^3|_{\D^2}, \rel\partial) \hookrightarrow \FEmb(\D^2, \F, (\NS^3, \xi), \rel \partial) \to \Emb(\D^2, \NS^3, \rel \partial),
    \end{equation}
    where the projection map is the natural forgetful map and $\Path_{(*, \U(1))}\operatorname{Iso}(T\NS^3|_{\D^2},\rel\partial)$ denotes the space of paths $F_s$, $s\in[0,1]$, of vector bundle isomorphisms that start at a point $F_0=*$ and end at
    $F_1\in\operatorname{Maps}(\D^2,\U(1))$,  
    that are fixed at the boundary of $\D^2$,
    where we see $\U(1)$ as a subset of $\operatorname{SO}(3)\subset\GL^+(3,\R)$.

    We can express the fiber as follows:
    \begin{align*}
        \Path_{(*, \U(1))}\operatorname{Iso}(T\NS^3|_{\D^2},\rel \partial) 
        & = \Path_{(*, \U(1))}\operatorname{Maps}(\D^2, \GL^+(3, \R),\rel \partial) \\
        & \cong \Path_{(*, \U(1))}\operatorname{Maps}(\NS^2, \GL^+(3, \R),\rel p) \\
        & \cong \Path_{(*, \U(1))}\operatorname{Maps}(\NS^2, \SO(3),\rel p),
    \end{align*}
    where $p\in \NS^2$.
    
    On the other hand, it is known that $\Path_{(*, \U(1))}\operatorname{Maps}(\NS^2, \SO(3), \rel p) = \Path_{(*, \U(1))}\Omega^2(\SO(3))$ projects to $\operatorname{Maps}(\NS^2, \U(1),\rel p) = \Omega^2(\U(1))$. Thus, we obtain the fibration
    $$
    \Omega^3(\SO(3)) \hookrightarrow \Path_{(*, \U(1))} \Omega^2(\SO(3)) \to \Omega^2(\U(1)),
    $$
    where the projection map consists on taking the endpoint of the path, and $\Omega^2(\U(1)) \cong \Omega^2(\NS^1)$, which is contractible. Consequently, $\Omega^3(\SO(3))$ is homotopically equivalent to the fiber of fibration (\ref{fibracion-Path}), and therefore we have the following long exact sequence:
    $$\xymatrix{  & \cdots\ar"1,3" & \pi_1(\Emb(\D^2, \NS^3, \rel \partial))\ar"2,1"\\ 
    \pi_0(\Omega^3(\SO(3)))\ar"2,2" & \pi_0(\FEmb(\D^2, \F, (\NS^3, \xi), \rel \partial))\ar"2,3" & \pi_0(\Emb(\D^2, \NS^3, \rel \partial)).}
    $$
    Since $\Emb(\D^2, \NS^3, \rel \partial)$ is contractible by Theorem \ref{thm:HatcherDisks}, and 
    $\pi_0( \Omega^3(\SO(3))) = \pi_3(\SO(3)) = \mathbb{Z}$, we conclude that:
    $$
    \pi_0(\Emb(\D^2, \F, (\NS^3, \xi), \rel \partial))\cong \pi_0(\FEmb(\D^2, \F, (\NS^3, \xi), \rel \partial))\cong\Z.
    $$

\end{proof}

\subsection{Contactomorphisms of an overtwisted $3$-sphere}

We start with the following observation 

\begin{lemma}\label{lem:Rotations}
    Let $(\NS^3,\xi)$ be an overtwisted contact $3$-sphere. Then, the natural $1$-jet evaluation map $\Cont(\NS^3,\xi)\rightarrow \CFra(\NS^3,\xi)$ is surjective at the level of fundamental groups.
\end{lemma}
\begin{proof}
    This follows from the fact that we can choose $\xi$ to be invariant under rotations around one complex plane in $\mathbb{C}^2\supseteq \NS^3$.
\end{proof}

Given a point $p\in \NS^3$ we will denote by $\Cont(\NS^3,\xi,\rel p)\subseteq \Cont(\NS^3,\xi)$ the subspace of contactomorphisms that are fixed near a small Darboux disk containing $p$. We will make use of the following

\begin{theorem}[Chekanov, Vogel \cite{VogelOvertwisted}]\label{thm:ContactomorphismS3}
    Let $(\NS^3,\xi)$ be an overtwisted contact $3$-sphere and $p\in \NS^3$ a given point. Then, 
  \[\pi_0(\Cont(\NS^3,\xi,\rel p)\cong \pi_0(\Cont(\NS^3, \xi)) \cong 
    \begin{cases}
        \Z_2 \oplus \Z_2 & \text{if } \xi \cong \xi_{-1}, \\
        \Z_2 & \text{otherwise}
         \end{cases}\]
\end{theorem}
\begin{proof}
    Without fixing the Darboux ball this is \cite[Theorem 4.5]{VogelOvertwisted}. The case in which the Darboux ball is fixed follows from this case by studying the fibration 
    $$ \Cont(\NS^3,\xi,\rel p) \hookrightarrow \Cont(\NS^3,\xi)\rightarrow \CFra(\NS^3,\xi) $$ and using Lemma \ref{lem:Rotations}.
\end{proof}

\subsection{Proof of Corollary \ref{cor:unknots}}

Let $\Lambda\subseteq (\D^3,\xi_{\std})\subseteq (M,\xi)$ be a Legendrian unknot that lies in a Darboux ball. Fix an embedded convex disk $\D^2\subseteq (M,\xi)$ with $\partial \D^2=\Lambda$ and tight characteristic foliation $\F$. We will consider the auxiliary spaces $\Emb(\D^2, \F, (M, \xi), \rel p)$ and $\FEmb(\D^2, \F, (M, \xi), \rel p)$ of embeddings and formal embeddings of disks with characteristic foliation $\F$ that are fixed near a point $p\in \Lambda=\partial \D^2$.
    
Consider the following commutative diagram in which the rows are Serre fibrations and the columns are the natural inclusions:
    $$\xymatrix{
    \Emb(\D^2, \F, (M, \xi), \rel \partial) \ar"1,2"\ar@{^(->}[d] & \Emb(\D^2, \F, (M, \xi), \rel p) \ar"1,3" \ar@{^(->}[d] & \mathcal{L}_{(p,v)}(\Lambda, (M, \xi))\ar@{^(->}[d]\\
     \FEmb(\D^2, \F, (M, \xi), \rel \partial) \ar"2,2" &  \FEmb(\D^2, \F, (M, \xi), \rel p) \ar"2,3" & F\mathcal{L}_{(p,v)}(\Lambda, (M, \xi))}$$
     Here, $F\mathcal{L}_{(p,v)}(\Lambda, (M, \xi))=\{(\gamma,F_s)\in F\mathcal{L}(\Lambda, (M, \xi)): \gamma(0)=p, F_s(0)=v, s\in[0,1]\}.$
    It was observed in \cite{FMP22} that $\Emb(\D^2, \F, (M, \xi), \rel p)$ is contractible. On the other hand, it is straightforward to check that $\FEmb(\D^2, \F, (M, \xi), \rel p)$ is also contractible. Therefore, there is a commutative square 
    $$\xymatrix{
    \Omega\mathcal{L}_{(p,v)}(\Lambda, (M, \xi)) \ar"1,2"\ar@{^(->}[d] & \Emb(\D^2, \F, (M, \xi), \rel \partial) \ar@{^(->}[d] \\
     \Omega F\mathcal{L}_{(p,v)}(\Lambda, (M, \xi)) \ar"2,2" &  \FEmb(\D^2, \F, (M, \xi), \rel \partial)}$$
     in which the horizontal maps are weak homotopy equivalences.

     Therefore, by Theorem \ref{thm:principal}, we conclude that the inclusion $\mathcal{L}_{(p,v)}(\Lambda, (M, \xi))\hookrightarrow F\mathcal{L}_{(p,v)}(\Lambda, (M, \xi))$ induces an isomorphism at the level of fundamental groups. 

     To conclude the proof we will consider the commutative diagram 

    $$\xymatrix{
    \mathcal{L}_{(p,v)}(\Lambda, (M, \xi)) \ar"1,2"\ar@{^(->}[d] & \mathcal{L}(\Lambda, (M, \xi)) \ar"1,3" \ar@{^(->}[d] & \CFra(M, \xi))\ar@{^(->}[d]\\
     \widetilde{F\mathcal{L}}_{(p,v)}(\Lambda, (M, \xi)) \ar"2,2" &  F\mathcal{L}(\Lambda, (M, \xi)) \ar"2,3" & \CFra(M, \xi))}$$
    induced by the evaluation map $F\mathcal{L}(\Lambda, (M, \xi))\rightarrow \CFra(M,\xi),(\gamma,F_s)\mapsto (\gamma(0),F_1(0)).$ Here, the rows are Serre fibrations and the colums are inclusions. Notice that the fiber $\widetilde{F\mathcal{L}}_{(p,v)}(\Lambda, (M, \xi))$ is weakly homotopy equivalent to $F\mathcal{L}_{(p,v)}(\Lambda, (M, \xi))$. The result now follows by analyzing the long exact sequences in homotopy and the Five Lemma since the homomorphism $$\pi_1( \mathcal{L}_{(p,v)}(\Lambda, (M, \xi)))\rightarrow \pi_1( F\mathcal{L}_{(p,v)}(\Lambda, (M, \xi)) ) $$ is an isomorphism as we just proved. This concludes the argument.
\qed

\subsection{Proof of Corollary \ref{cor:UnknotsS3}}
   Consider the fibration 

    \begin{equation}\label{eq:evaluationLegS3} \mathcal{L}_{(p,v)}(\Lambda, (\NS^3, \xi)) \hookrightarrow \mathcal{L}(\NS^3,\xi)\rightarrow \CFra(\NS^3,\xi) 
    \end{equation}

    Notice that $\CFra(\NS^3,\xi)\cong \NS^3\times \U(1)$ since every oriented plane bundle over $\NS^3$ is trivializable. In particular, $\pi_1( \CFra(\NS^3,\xi) )\cong \Z $ and $\pi_2 (\CFra(\NS^3,\xi))=0$. Moreover, as explained in the proof of Corollary \ref{cor:unknots} there is an isomorphism $$ \pi_1( \mathcal{L}_{(p,v)}(\Lambda, (\NS^3, \xi)) ) \cong \pi_0( \Emb(\D^2, \F, (M, \xi), \rel \partial) ) \cong \Z, $$
    where in the last isomorphism we have used Corollary \ref{cor:DisksS3}. Using all this information in the long exact sequence in homotopy associated to the fibration (\ref{eq:evaluationLegS3}) we conclude that there is a short exact sequence 

    $$ 0\mapsto \Z\cong  \pi_1( \mathcal{L}_{(p,v)}(\Lambda, (\NS^3, \xi)) ) \rightarrow \pi_1 ( \mathcal{L}(\Lambda, (\NS^3, \xi)) ) \rightarrow \Z \rightarrow 0 $$

    From here it follows that $\pi_1 ( \mathcal{L}(\Lambda, (\NS^3, \xi)) )$ is generated by two independent loops of Legendrian unknots as described in Corollary \ref{cor:UnknotsS3}. One loop $A$ represented by a full-rotation of $\Lambda$ in a Darboux chart and the other loop $B$ represented by a loop of long Legendrian unknots, which is contractible as a loop of long smooth unknots by a result of Hatcher \cite{Smale}. Moreover, it follows that $\pi_1 ( \mathcal{L}(\Lambda, (\NS^3, \xi)) )$ is isomorphic to either $\Z\oplus \Z$ or $\Z \rtimes \Z$\footnote{Recall that, since $\operatorname{Aut}(\Z)=\{\Id,-\Id\}$, there are exactly two semidirect products of $\Z$ with itself, and only one of the two is commutative, which is precisely the trivial product.}. Since the latter group is non abelian it enough to check that the loops $A$ and $B$ commute to conclude the proof. 

    By Lemma \ref{lem:Rotations} the map $\pi_1(\Cont(\NS^3,\xi))\rightarrow \pi_1(\CFra(\NS^3,\xi))$ is surjective. Therefore, there exists a loop of contactomorphisms $\varphi^t\in \Cont(\NS^3,\xi)$, $t\in [0,1]$, $\varphi^0=\varphi^1=\Id$, such that $A=\varphi^t(\Lambda)$. We will further assume that $\varphi^t=\Id$ for $t\in [0,\frac{1}{2}]$. On the other hand, we may represent $B=\Lambda ^ t$, $t\in [0,1]$, by a loop of long Legendrians such that $\Lambda ^t = \Lambda$ for all $t\in [\frac{1}{2},1]$. It then follows that the concatenation $A*B$ can be represented by the loop $\varphi^t(\Lambda^t)$, $t\in [0,1]$, it is now straightforward to check that $A*B=B*A$ as elements in $\pi_1(\mathcal{L}(\Lambda, (\NS^3, \xi)))$. This concludes the argument.
\qed

\subsection{Proof of Corollary \ref{cor:pi0complementounknot}}

Let $\Lambda\subseteq (\NS^3,\xi)$ be some Legendrian unknot with non-positive $\tb$ in an overtwisted contact $3$-sphere. We fix an embedded convex disk $\D^2\subseteq (\NS^3,\xi)$ with boundary $\partial \D^2=\Lambda$. In the case in which $\tb(\Lambda)+|\rot(\Lambda)|\leq -1$ we will assume that the characteristic foliation $\F$ of $\D^2$ is tight. We will make the identification $\Cont(C(\Lambda),\xi)=\Cont(\NS^3,\xi,\rel \Lambda)$.

We consider first the case in which $\tb(\Lambda)+|\rot(\Lambda)|>-1$. Consider the natural fibration 
\begin{equation}\label{eq:fibrationContDisks}
\Cont(\NS^3,\xi,\rel \D^2)\hookrightarrow \Cont(\NS^3,\xi,\rel \Lambda)\rightarrow \Emb(\D^2,\F,(\NS^3,\xi),\rel \partial).
\end{equation}
given by post-composition of the inclusion $\D^2\hookrightarrow (\NS^3,\xi)$ with a contactomorphism. 

We claim that the fibration map is surjective in homotopy groups. Indeed, this is an application of Eliashberg's $h$-principle (Theorem \ref{thm:h-OT}). We explain the non-parametric case since the general case is analogous. Given two elements $e,f\in \Emb(\D^2,\F,(\NS^3,\xi),\rel \partial)$ it is not hard to find an orientation preserving diffeomorphism $\varphi\in \Diff(\NS^3)$ such that $f=\varphi\circ e$, since the space of smooth embeddings is connected by Schoenflies Theorem (in the parametric case, we use that this space is contractible by Theorem 
\ref{thm:HatcherDisks}). Since the embeddings $e$ and $f$ have the same characteristic foliation we may assume that $\varphi$ is an actual contactomorphism over $(\Op(\image(e)),\xi)$. The contact structures $\varphi_* \xi$ and $\xi$ are homotopic as formal contact structures relative $\Op(\image(f))$ (in the parametric case, we use that the obtained spheres of formal contact structures are contractible because the group $\Diff(\NS^3,\rel \Lambda)$ is contractible by a combination of Theorem \ref{thm:HatcherDisks} and the Hatcher's proof of the Smale conjecture \cite{Smale}\footnote{This appears as one of the equivalent forms of the Smale conjecture at the end of \cite{Smale}. To prove this equivalence Theorem \ref{thm:HatcherDisks} is needed.}, so it 
acts trivially in homotopy on the space of formal contact structures). Since $(\Op(\image(f)),\xi)$ is overtwisted, by the assumption $\tb(\Lambda)+|\rot(\Lambda)|>-1$, we can use Theorem \ref{thm:h-OT} to find a homotopy of contact structures $\xi^t\in \CStr(\NS^3)$, $t\in [0,1]$, such that $\xi^0=\varphi_*\xi$, $\xi^1=\xi$ and 
$\xi^t_{|\Op(\image(f))}=\xi_{|\Op(\image(f))}$ for all $t\in [0,1]$. By Gray stability there exists an isotopy $g^t\in \Diff(\NS^3,\rel \image(f))$, $t\in [0,1]$, such that $g^t_*\xi^0=\xi^t$. The diffeomorphism $G=g^1\circ \varphi\in \Cont(\NS^3,\xi,\rel\Lambda)$ and satisfies $G_* \xi=\xi$, as required. 

On the other hand, it was observed by Dymara in \cite{Dymara} that there is a weak homotopy equivalence $\Cont(\NS^3,\xi,\rel \D^2)\cong \FCont(\NS^3,\xi,\rel \D^2)\cong\Omega^4 \NS^2$. Here, we are using again that $(\Op(\D^2),\xi)$ is overtwisted so $(\NS^3\backslash\D^2,\xi)$ is overtwisted at infinity (see Theorem \ref{thm:ContOT}). In particular, $$ \pi_0(\Cont(\NS^3,\xi,\rel \D^2))\cong\pi_0(\FCont(\NS^3,\xi,\rel \D^2))\cong\pi_4(\NS^2)\cong\Z_2.$$ On the other hand, by Theorem \ref{thm:principal} and Corollary \ref{cor:DisksS3} there is an isomorphism 
$$ \pi_0( \Emb(\D^2,\F,(\NS^3,\xi),\rel \partial) )\cong \pi_0(\FEmb(\D^2,\F,(\NS^3,\xi),\rel \partial) ) \cong \Z.$$

Hence, it follows from the short exact sequence in homotopy associated to the fibration (\ref{eq:fibrationContDisks}) that there is a short exact sequence 
$$ 0\rightarrow \Z_2\rightarrow \pi_0(\Cont(\NS^3,\xi,\rel \Lambda))\rightarrow \Z\rightarrow 0 $$
from which it follows that $$ \pi_0(\Cont(C(\Lambda),\xi))=\pi_0(\Cont(\NS^3,\xi,\rel \Lambda))=\Z\oplus \Z_2. $$
This concludes the argument in the case that $\tb(\Lambda)+|\rot(\Lambda)|>-1$.

We now deal with the case $\tb(\Lambda)+|\rot(\Lambda)|\leq -1$. Consider the fibration
$$ \Cont(\NS^3,\xi,\rel \Lambda)\hookrightarrow \Cont(\NS^3,\xi,\rel p) \rightarrow \mathcal{L}_{(p,v)}(\Lambda,(\NS^3,\xi)) $$
given by post-composition of the inclusion $\Lambda\hookrightarrow (\NS^3,\xi)$ with a given contactomorphism. Observe that the fibration map factors as 
$$ \Cont(\NS^3,\xi,\rel p) \rightarrow \Emb(\D^2,\F,(\NS^3,\xi),\rel p) \rightarrow \mathcal{L}_{(p,v)}(\Lambda,(\NS^3,\xi)). $$
Therefore, since the space $\Emb(\D^2,\F,(\NS^3,\xi),\rel p)$ is contractible because $\F$ is tight by \cite{FMP22} we conclude that the fibration map is null-homotopic. In particular, there are weak homotopy equivalences 
$$ \Cont(\NS^3,\xi,\rel \Lambda)\cong \Cont(\NS^3,\xi,\rel p)\times \Omega \mathcal{L}_{(p,v)}(\Lambda,(\NS^3,\xi)) \cong \Cont(\NS^3,\xi,\rel p)\times \Emb(\D^2,\F,(\NS^3,\xi),\rel \partial ),$$ where the last one was explained in the proof of Corollary \ref{cor:unknots}. The result now follows by combining Corollary \ref{cor:DisksS3} and Theorem \ref{thm:ContactomorphismS3}.
\qed

\bibliographystyle{plain}
\bibliography{main}

\end{document}